\documentclass[preprint,sort&compress,12pt]{elsarticle}
\usepackage{amscd,amssymb,amsmath,amsthm}
\usepackage[all]{xy}
\usepackage{array}
\usepackage{etoolbox}
\usepackage{mathtools}
\DeclarePairedDelimiter{\ceil}{\lceil}{\rceil}
\DeclarePairedDelimiter{\floor}{\lfloor}{\rfloor}
\makeatletter
\patchcmd{\ps@pprintTitle}{\footnotesize\itshape
       Preprint submitted to \ifx\@journal\@empty Elsevier
       \else\@journal\fi\hfill\today}{\relax}{}{}
\makeatother

\newdir{ >}{!/8pt/\dir{}*\dir{>}}

\newtheorem{theorem}{Theorem}[section]

\newtheorem{lemma}[theorem]{Lemma}
\newtheorem{corollary}[theorem]{Corollary}
\newtheorem{prop}[theorem]{Proposition}

\newtheorem*{con}{Conjecture}
\newtheorem{remark}[theorem]{Remark}
\theoremstyle{definition}
\newtheorem{definition}[theorem]{Definition}
\newtheorem{example}[theorem]{Example}

\newtheorem*{th39}{Theorem 3.9}
\newtheorem*{th311}{Theorem 3.11}
\newtheorem*{th46}{Theorem 4.6}
\newtheorem*{th55}{Theorem 5.5}
\newtheorem*{th56}{Theorem 5.6}
\newtheorem*{th61}{Theorem 6.1}
\newtheorem*{th65}{Theorem 6.5}
\newtheorem*{th73}{Theorem 7.3}

\numberwithin{equation}{theorem}

\DeclareMathOperator{\Ker}{\mathit{ker}}
\DeclareMathOperator{\im}{\mathit{im}}

\DeclareMathOperator{\M}{\mathit{M(G)}}
\DeclareMathOperator{\MM}{\mathit{M}}
\DeclareMathOperator{\e}{\mathit{exp}}

\DeclareMathOperator{\Z}{\mathit{Z(G)}}
\DeclareMathOperator{\ZZ}{\mathit{Z}}
\DeclareMathOperator{\p}{\mathit{p}}
\DeclareMathOperator{\cc}{\mathit{c}}
\DeclareMathOperator{\dd}{\mathit{d}}
\DeclareMathOperator{\G}{\mathit{G}}
\DeclareMathOperator{\N}{\mathit{N}}
\DeclareMathOperator{\HH}{\mathit{H}}
\DeclareMathOperator{\E}{\mathit{E}}

\begin{document}

\begin{frontmatter}

\title{ON THE EXPONENT CONJECTURE OF SCHUR}

 \author[IISER TVM]{A.E. Antony}
\ead{ammu13@iisertvm.ac.in}
\author[IISER TVM]{P. Komma}
\ead{patalik16@iisertvm.ac.in}
\author[IISER TVM]{V.Z. Thomas\corref{cor1}}
\address[IISER TVM]{School of Mathematics,  Indian Institute of Science Education and Research Thiruvananthapuram,\\695551
Kerala, India.}
\ead{vthomas@iisertvm.ac.in}
\cortext[cor1]{Corresponding author. \emph{Phone number}: +91 8078020899.}

\begin{abstract}
It is a longstanding conjecture that for a finite group $\G$, the exponent of the second homology group $\HH_2(\G, \mathbb{Z})$ divides the exponent of $\G$. In this paper, we prove this conjecture for $\p$-groups of class at most $\p$, finite nilpotent groups of odd exponent and of nilpotency class 5, $\p$-central metabelian $\p$-groups, and groups considered by L. E . Wilson in \cite{LEW}. Moreover, we improve several bounds given by various authors. We achieve most of our results using an induction argument.

\end{abstract}


\end{frontmatter}

\section{Introduction}
The Schur multiplier of a group $\G$, denoted by $\M$ is the second homology group of $\G$ with coefficients in $\mathbb{Z}$, i.e $\M=\HH_2(\G, \mathbb{Z})$. A longstanding conjecture attributed to I. Schur says that for a finite group $\G$,
\begin{equation}\label{E1}
\tag{1}  \e(\M)|\e(\G). 
\end{equation}

To prove ($\ref{E1}$), it is enough to restrict ourselves to $\p$-groups using a standard argument given in Theorem 4, Chapter IX of \cite{JPS}. A. Lubotzky and A. Mann showed that (\ref{E1}) holds for powerful $\p$-groups(\cite{LM}), M. R. Jones in \cite{MRJ} proved that (\ref{E1}) holds for groups of class 2,  P. Moravec showed that the conjecture holds for groups of nilpotency class at most 3, groups of nilpotency class 4 and of odd order, potent $\p$-groups, metabelian $\p$-groups of exponent $\p$, $\p$-groups of class at most $\p-2$ (\cite{PM1}, \cite{PM2}, \cite{PM3}, and \cite{PM4}) and some other classes of groups. The general validity of $(\ref{E1})$ was disproved by A. J. Bayes, J. Kautsky and J. W. Wamsley in \cite{BKW}. Their counterexample involved a 2-group of order $2^{68}$ with $\e(G)=4$ and $\e(\M)=8$. Nevertheless for finite groups of odd exponent, this conjecture remains open till date. This problem has remained open even for finite $\p$-groups of class 5 having odd exponent. The purpose of this paper is to prove (\ref{E1}) for finite $\p$-groups of class at most $\p$ and finite $\p$-groups of class 5 with odd exponent. Moreover we also prove the above mentioned results of \cite{LM}, \cite{PM1}, \cite{PM2}, \cite{PM3} and \cite{PM4} for odd primes and hence proving all these results using a common technique and bringing them under one umbrella. We briefly describe the organization of the paper by listing the main results according to their sections.

In \cite{IS}, I. Schur proves that if $\Z$ has finite index in $G$, then the commutator subgroup $\gamma_2(G)$ is finite. In the next Theorem, we generalize this classical theorem of Schur for a $\p$-group $\G$ of class at most $\p+1$. We use this generalization to prove the conjecture for $\p$-groups of class at most $\p$.

\begin{th39}
Let $\p$ be an odd prime and $\G$ be a $\p$-group of nilpotency class $\p+1$. If $\G$ is $\p^n$-central, then $\e( \gamma_2(\G)) \mid \p^n$.
\end{th39}

In \cite{PM3}, the author shows that $(\ref{E1})$ holds for $\p$-groups of class less than or equal to $\p-2$. Authors of \cite{MHM1} prove the same for groups of class less than or equal to $\p-1$. In the next Theorem, we generalize both the above results by proving

\begin{th311}
Let $\p$ be an odd prime and $\G$ be a finite $\p$-group. If the nilpotency class of $\G$ is at most $\p$, then $\e(\G \wedge \G)\mid \e(\G)$. In particular, $\e(\M)\mid \e(\G)$.
\end{th311}

In \cite{PM5}, the author proves that if $\G$ is a group of nilpotency class 5, then $\e(\M)\mid (\e(\G))^2$. In the next Theorem, we improve this bound and in fact prove the conjecture for groups with odd exponent.

\begin{th46}
Let $\G$ be a finite $\p$-group of nilpotency class 5. If $\p$ is odd, then $\e(\G\wedge \G)\mid \e(\G)$. In particular, $\e(\M)\mid \e(\G)$.
\end{th46}

In Section $5$, we prove two main theorems which we list below. The second condition in the next Theorem generalizes the definition of powerful $2$-groups for odd primes and were considered by L. E. Wilson in \cite{LEW}, and the first condition includes the class of groups considered by D. Arganbright in \cite{DA}, and it also includes the class of potent $\p$-groups considered by J. Gonzalez-Sanchez and  A. Jaikin-Zapirain in \cite{SZ}. Thus as a corollary of the next Theorem, we obtain the well-known result that (\ref{E1}) holds for powerful $\p$-groups (\cite{LM}) and potent $\p$-groups (\cite{PM4}). 

\begin{th55}
Let $\p$ be an odd prime and $\G$ be a finite $\p$-group satisfying either of the conditions below,
\begin{itemize}
\item[$(i)$] $\gamma_m(\G)\subset \G^{\p}$ for some  $m$ with $2\leq m\leq \p-1$.
\item[$(ii)$] $\gamma_{\p}(\G)\subset \G^{\p^2}$.
\end{itemize}
 Then $\e(\G\wedge \G)\mid \e(\G)$. In particular, $\e(\M)\mid \e(\G)$.
\end{th55}

In \cite{NS2}, the author has shown that proving the conjecture for regular $\p$-groups is equivalent to proving it for groups with exponent $\p$. As a corollary to the next Theorem, we obtain the same.

 \begin{th56}
The following statements are equivalent:
\begin{itemize}
\item[$(i)$] $\e(\G\wedge \G)\mid \e(\G)$ for all regular $\p$-groups $\G$.
\item[$(ii)$] $\e(\G\wedge \G)\mid \e(\G)$ for all groups $\G$ of exponent $\p$.
\end{itemize}
 \end{th56}

In section $6$, we give bounds on the exponent of $\M$ that depend on nilpotency class. G. Ellis in \cite{GE2} showed that if $G$ is a group with nilpotency class $\cc$, then $\e(\M)\mid ( \e(\G))^{\ceil{\frac{\cc}{2}}}$. P. Moravec in \cite{PM1} improved this bound by showing that $\e(\M)\mid ( \e(\G))^{2(\floor{\log_2 \cc})}$. In the next Theorem, we improve the bounds given in \cite{GE2} and \cite{PM1}.

\begin{th61}
Let $\G$ be a finite group with nilpotency class $\cc>1$ and set  $n= \ceil{\log_{3}(\frac{\cc+1}{2})}$. If $\e(\G)$ is odd, then $\e(\G\wedge \G) \mid (\e(\G))^n$. In particular, $\e(\M)\mid (\e(\G))^n$.
\end{th61}

In Theorem 1.1 of \cite{NS2}, N. Sambonet improved all the bounds obtained by various authors by proving that $\e(\M)\mid (\e(\G))^m$, where $m=\floor{\log_{\p-1} \cc}+1$. We improve the bound given in \cite{NS2} by proving, 

\begin{th65}
Let $\p$ be an odd prime and $\G$ be a finite $\p$-group of nilpotency class $\cc\ge \p$. Then $\e(\G\wedge \G)\mid (\e(\G))^n$, where $n = 1+\ceil{\log_{\p-1} (\frac{\cc+1}{\p+1})}$. In particular, $\e(\M)\mid (\e(\G))^n$.
\end{th65}

For a solvable $\p$-group of derived length $\dd$, the author of \cite{NS1} proves that if $\p$ is odd, then  $\e(\M)\mid (\e(\G))^{\dd}$, and if $\p=2$, then $\e(\M)\mid 2^{\dd-1}(\e(\G))^{\dd}$. Using our techniques, we obtain the following generalization of Theorem A of \cite{NS1}, which is one of their main results. 

\begin{th73}
Let $\G$ be a solvable group of derived length $\dd$. 
\begin{itemize}
\item[$(i)$] If $\e(\G)$ is odd, then $\e(\G\otimes \G)\mid (\e(\G))^{\dd}$. In particular, $\e(\M)\mid (\e(\G))^{\dd}$.
\item[$(ii)$] If $\e(\G)$ is even, then $\e(\G\otimes \G)\mid 2^{\dd-1}(\e(\G))^{\dd}$. In particular, $\e(\M)\mid 2^{\dd-1}(\e(\G))^{\dd}$.
\end{itemize}
\end{th73}

\section{Preparatory Results}

R. Brown and J.-L. Loday introduced the nonabelian tensor product $\G\otimes \HH$ for a pair of groups $\G$ and $\HH$ in \cite{BL1} and \cite{BL2} in the context of an application in homotopy theory, extending the ideas of J.H.C. Whitehead in \cite{W}. A special case, the nonabelian tensor square, already appeared in the work of R.K. Dennis in \cite{RKD}. The nonabelian tensor product of groups is defined for a pair of groups that act on each other provided the actions satisfy the compatibility conditions of Definition \ref{D:2.1} below. Note that we write conjugation on the left, so $^gg'=gg'g^{-1}$ for $g,g'\in \G$ and $^gg'g'^{-1}=[g,g']$ for the commutator of $g$ and $g'$. Moreover, our commutators are right normed, i.e $[a,b,c]=[a,[b,c]]$.

\begin{definition}\label{D:2.1}
Let $\G$ and $\HH$ be groups that act on themselves by conjugation and each of which acts on the other. The mutual actions are said to be compatible if
\begin{equation}
^{^h g}h'=\; ^{hgh^{-1}}h' \;\mbox{and}\; ^{^g h}g'=\ ^{ghg^{-1}}g' \;\mbox{for \;all}\; g,g'\in \G \mbox{and}\;h,h'\in \HH.
\end{equation}
\end{definition}

\begin{definition}\label{D:2.2}
Let $\G$ be a group that acts on itself by conjugation, then the nonabelian tensor square $\G\otimes \G$ is the group generated by the symbols $g\otimes h$ for $g,h\in \G$  with relations
\begin{equation}\label{eq:2.2.1}
gg'\otimes h=(^gg'\otimes \;^gh)(g\otimes h),  
\end{equation}
\begin{equation}\label{eq:2.2.2}
g\otimes hh'=(g\otimes h)(^hg\otimes \;^hh'),   
\end{equation}
\noindent for all $g,g',h,h'\in \G $.
\end{definition}

There exists a homomorphism $\kappa : \G\otimes \G \rightarrow \G^{\prime}$ sending $g\otimes h$ to $[g,h]$. Let $\nabla (\G)$ denote the subgroup of $\G\otimes \G$ generated by the elements $x\otimes x$ for $x\in \G$. The exterior square of $\G$ is defined as $\G\wedge \G= (\G\otimes \G)/\nabla (\G)$. We get an induced homomorphism $\G\wedge \G \rightarrow \G^{\prime}$, which we also denote as $\kappa$. We set $\M$ as the kernel of this induced homomorphism, it has been shown in \cite{MC} that $\M\cong \HH_{2}(\G, \mathbb{Z})$, the second homology group of $\G$.

We can find the following results  in \cite{BJR} and Proposition 3 of \cite{VT}.

\begin{prop}\label{P:2.3}
\begin{itemize}
\item[(i)] There are homomorphisms of groups $\lambda : \G \otimes \HH \rightarrow \G,\ \lambda': \G \otimes \HH \rightarrow \HH$ such that $\lambda(g \otimes h) = g^hg^{-1},\ \lambda'(g \otimes h)=\; ^ghh^{-1}$.
\item[(ii)] The crossed module rules hold for $\lambda$ and $\lambda'$, that is,
\begin{align*}
\lambda(^gt) &=\ g(\lambda(t))g^{-1},\\
tt_1t^{-1} &=\ ^{\lambda(t)}t_1,
\end{align*}
for all $t,t_1 \in \G \otimes \HH, g \in \G$ (and similarly for $\lambda'$).
\item[(iii)] $\lambda(t) \otimes h =\ t^ht^{-1},\ g \otimes \lambda'(t) =\  ^gtt^{-1}$, and thus $\lambda(t) \otimes \lambda'(t_1) = [t,t_1]$ for all $t,t_1 \in \G \otimes \HH, g \in \G, h \in \HH$. Hence $\G$ acts trivially on $\Ker \lambda'$ and $\HH$ acts trivially on $\Ker \lambda$.
\end{itemize}
In particular, the following relations hold for $g, g_1 \in \G$ and $h, h_1 \in \HH$ :
\begin{itemize}
\item[(iv)]\begin{equation}\label{eq:2.3.1}
^g(g^{-1} \otimes h) = (g \otimes h)^{-1} =\ ^h(g \otimes h^{-1}). 
\end{equation} 

\item[(v)]\begin{equation}\label{eq:2.3.2}
(g \otimes h)(g_1 \otimes h_1)(g \otimes h)^{-1} = (^{[g,h]}g_1 \otimes \ ^{[g,h]}h_1). 
\end{equation}

\item[(vi)]\begin{equation}\label{eq:2.3.3}
[g,h] \otimes h_1 = (g \otimes h)^{h_1} (g\otimes h)^{-1}.
\end{equation}

\item[(vii)]\begin{equation}\label{eq:2.3.4}
 g_1 \otimes\ [g,h] =\ ^{g_1}(g \otimes h) (g\otimes h)^{-1}. 
\end{equation}

\item[(viii)]\begin{equation}\label{eq:2.3.5}
 [g\otimes h, g_1 \otimes h_1] =  [g,h] \otimes [g_1,h_1].
\end{equation}

\end{itemize}
\end{prop}

Let $\N\unlhd \G$ and $[n, n, g]=1$ for all $g\in \G, n\in \N$. In \cite{GE2}, G. Ellis proves that for all integers $t\ge 2$, $n^t\otimes g = (n\otimes g)^t (n\otimes [n, g]^{{{t}\choose{2}}})$. In the next lemma, we will generalize this identity.

\begin{lemma}\label{L:2.4}
Let $\N, \MM \unlhd \G$. If $n\otimes [n, n, m]=1$ and $n\otimes m, n\otimes [n, m]$ commute for all $n\in \N, m\in \MM$, then $n^t\otimes m = (n\otimes [n, m])^{{{t}\choose{2}}} (n\otimes m)^t = (n^{{{t}\choose{2}}} \otimes [n, m]) (n\otimes m)^t$, for every $t\ge 2$.
\end{lemma}
\begin{proof}
We first prove that for every $t$, $n^t\otimes m = (n\otimes [n, m])^{{{t}\choose{2}}} (n\otimes m)^t$. Note that $n^2\otimes m =\ ^n(n\otimes m) (n\otimes m)$ and using \eqref{eq:2.3.4}, we obtain $^n(n\otimes m) = (n\otimes [n, m]) (n\otimes m)$. Thus $n^2\otimes m = (n\otimes [n, m]) (n\otimes m)^2$, and now we proceed by induction. Let $t > 2$ and assume the statement for $t-1$. We have $n^t\otimes m =\ ^n(n^{t-1}\otimes m) (n\otimes m) = \ ^n(n\otimes [n, m])^{{{t-1}\choose{2}}} \ ^n(n\otimes m)^{t-1} (n\otimes m)$. Since $n\otimes [n, n, m] = 1$, applying \eqref{eq:2.3.4} yields $^n(n\otimes [n, m])=n\otimes [n, m]$. Thus we obtain
\begin{align*}
{n^t}\otimes m&=(n\otimes [n, m])^{{t-1}\choose{2}}((n\otimes [n, m])(n\otimes m))^{t-1}(n\otimes m)
\\&=(n\otimes [n, m])^{{t-1}\choose{2}}(n\otimes [n, m])^{t-1}(n\otimes m)^{t-1}(n\otimes m)\\
&=(n\otimes [n, m])^{{{t}\choose{2}}}(n\otimes m)^t,
\end{align*}
which completes the inductive step. Since $^n(n\otimes [n, m]) = n\otimes [n, m]$, it follows that $n^{{{t}\choose {2}}}\otimes [n, m] = (n\otimes [n, m])^{{{t}\choose{2}}}$, which in turn yields $n^t\otimes m = (n^{{{t}\choose{2}}} \otimes [n, m]) (n\otimes m)^t$.
\end{proof}

\begin{lemma}\label{L:2.5}
Let $\N, \MM$ be normal subgroups of a group $\G$. If $\N$ is abelian, then the following holds:

\begin{itemize}
\item[$(i)$] For all $n_i\in \N$, $m_i \in \MM$, $i = 1, 2$ and for all $t\ge 2$, 
\begin{align*} 
((n_1\otimes m_1)(n_2\otimes m_2))^t &=([n_2, m_2]\otimes [n_1, m_1])^{{t}\choose{2}}(n_1\otimes m_1)^t(n_2\otimes m_2)^t\\ &=([n_2, m_2]^{{{t}\choose{2}}}\otimes [n_1, m_1])(n_1\otimes m_1)^t(n_2\otimes m_2)^t.
\end{align*}
\item[$(ii)$] If $\e(\N)$ is odd, then $\e(\N\otimes \MM)\mid \e(\N)$.
\item[$(iii)$] If $\e(\N)$ is even, then $\e(\N\otimes \MM)\mid 2\e(\N)$.
 \end{itemize}
\end{lemma}
\begin{proof}
Let $\e(\N) =e$.
\begin{itemize}
\item[$(i)$] Since $\N$ is abelian, it follows from  \cite{DLT} that the nilpotency class of $\N\otimes \MM$ is at most $2$. Hence, 
\begin{align*}
((n_1\otimes m_1)(n_2\otimes m_2))^t&= [(n_2\otimes m_2), (n_1\otimes m_1)]^{{t}\choose{2}}(n_1\otimes m_1)^t(n_2\otimes m_2)^t\\
&=([n_2, m_2]\otimes [n_1, m_1])^{{t}\choose{2}}(n_1\otimes m_1)^t(n_2\otimes m_2)^t .
\end{align*}
As $\N$ is abelian $^{[n_2, m_2]} ([n_2, m_2]\otimes [n_1, m_1]) = [n_2, m_2]\otimes [n_1, m_1]$, yielding $[n_2, m_2]^{{{t}\choose{2}}}\otimes [n_1, m_1] = ([n_2, m_2]\otimes [n_1, m_1])^{{{t}\choose{2}}}$. Thus $((n_1\otimes m_1)(n_2\otimes m_2))^{t}=([n_2, m_2]^{{{t}\choose{2}}} \otimes [n_1, m_1])(n_1\otimes m_1)^t(n_2\otimes m_2)^t$.
\item[$(ii)$] Because $\N$ is abelian, we have $[n, n, m] = 1$. Hence $n\otimes [n, n, m] =1$ and $[(n\otimes m), (n\otimes [n, m])] = [n, m]\otimes [n, n, m] = 1$. Thus by Lemma \ref{L:2.4},  we obtain that ${n_1^e}\otimes m_1 =(n_1^{{{e}\choose{2}}} \otimes [n_1, m_1])(n_1\otimes m_1)^e$. Since $e$ is odd, it follows that $e\mid {{e}\choose{2}}$, whence $(n_1\otimes m_1)^e=1$. By $(i)$, we have that $((n_1\otimes m_1)(n_2\otimes m_2))^e =([n_2, m_2]^{{{e}\choose{2}}}\otimes [n_1, m_1])(n_1\otimes m_1)^e(n_2\otimes m_2)^e =1$, which proves $\e(\N\otimes \MM)\mid e$.

\item[$(iii)$]  By Lemma \ref{L:2.4}, we obtain ${n_1^{2e}}\otimes m_1 = (n_1^{{2e}\choose{2}}\otimes [n_1, m_1])(n_1\otimes m_1)^{2e}$. Since $e\mid {{2e}\choose{2}}$, it follows that $(n_1\otimes m_1)^{2e}=1$. By $(i)$ we obtain $((n_1\otimes m_1)(n_2\otimes m_2))^{2e}=([n_2, m_2]^{{{2e}\choose{2}}}\otimes [n_1, m_1])(n_1\otimes m_1)^{2e}(n_2\otimes m_2)^{2e}$, which proves $\e(\N\otimes \MM)\mid 2\e(\N)$.
\end{itemize}
\end{proof}

In \cite{R} and \cite{EL}, the authors give an isomorphism between the nonabelian tensor square of $\G$ and the subgroup $[\G, \G^{\phi}]$ of $\eta(\G)$.  We use this isomorphism in the proof of the next lemma.
\begin{lemma}\label{L:2.6}
Let $\G$ be a nilpotent group of class $\cc$ and set $m = \ceil{\frac{\cc+1}{3}}$. If $\gamma_m(\G)$ has odd exponent, then $\e(\im(\gamma_m(\G) \otimes \G))) \mid \e(\gamma_m(\G))$.
\end{lemma}
\begin{proof} Let $\e(\gamma_m(\G)) = e$.
Consider the isomorphism $\psi : \G\otimes \G \rightarrow [\G,\G^{\phi}]$ defined by $\psi(g\otimes h)= [g,h^{\phi}]$, where $\G^{\phi}$ is an isomorphic copy of $\G$. Recall that
$[\G,\G^{\phi}]$ is a subgroup of $\eta(\G)$. By Theorem A of \cite{R}, $\eta(\G)$ is a group of nilpotency class at most $\cc+1$. Note that for $n\in \gamma_m(\G), g\in \G$, $\psi(n\otimes [n, n, g])\in \gamma_{3m+1}(\eta(\G))$. Since $3m+1\ge \cc+2$, $\gamma_{3m+1}(\eta(\G)) \subset \gamma_{\cc+2}(\eta(\G))=1$. Thus $n\otimes [n, n, g] = 1$. Similarly $\psi([n, g]\otimes [n, n, g])\subset \gamma_{3m+2}(\eta(\G))$, giving $[(n\otimes g), (n\otimes [n, g])]=[n, g]\otimes [n, n, g] = 1$. Thus by Lemma \ref{L:2.4}, we obtain $n^e\otimes g = (n^{{{e}\choose{2}}} \otimes [n, g]) (n\otimes g)^e$. Since $e$ is odd, it follows that $e\mid {{e}\choose{2}}$, whence $(n\otimes g)^e=1$. Moreover for $n_1,n_2 \in \gamma_m(\G)$ and $g_1,g_2 \in \G$, we have $\psi( \gamma_3(\langle n_1\otimes g_1,n_2\otimes g_2\rangle))\subset \gamma_{3m+3}(\eta(\G)) =1$. But $3m+3\ge \cc+2$, so $\gamma_{3m+3}(\eta(\G))=1$. Hence $\langle n_1\otimes g_1,n_2\otimes g_2\rangle$ is a group of class at most 2. Thus $((n_1\otimes g_1)(n_2\otimes g_2))^e =([n_2, g_2]\otimes [n_1, g_1])^{{e}\choose{2}}(n_1\otimes g_1)^e(n_2\otimes g_2)^e =1$, proving the result.
\end{proof}

In the next lemma, we collect some commutator identities.

\begin{lemma}\label{L:2.7}
\begin{itemize}
\item[$(i)$] Let $\G$ be a group and $g_i, h_i\in \G$ for $i=1, 2$. Then
\begin{equation} \label{eq:2.7.1}
[gg_1, h] =\ ^g[g_1, h] [g, h] = [g, g_1, h] [g_1, h] [g, h],
\end{equation}
\begin{equation} \label{eq:2.7.2}
[g, hh_1] = [g, h]\ ^h[g, h_1] = [g, h] [h, g, h_1] [g, h_1].
\end{equation}
\item[$(ii)$] Let $\G$ be a group of nilpotency class $\cc$, $x\in \gamma_i(\G)$, $y\in \gamma_j(\G)$ and $z\in \gamma_k(\G)$. If $i+j+k \ge \cc+1$, then
\begin{equation} \label{eq:2.7.3} 
^x[y, z] = [y, z].
\end{equation}
\item[$(iii)$] Let $\G$ be a group of nilpotency class $\cc$, $x\in \gamma_i(\G)$, $y\in \gamma_j(\G)$, $z\in \gamma_k(\G)$ and $u\in \gamma_l(\G)$. If $i+j+k+l \ge \cc+1$, then
\begin{equation} \label{eq:2.7.4}
[x, y] [z, u] = [z, u] [x, y]
\end{equation}
\end{itemize}
\end{lemma}

\begin{proof}
Identities in $(i)$ can be found in any standard book on group theory and $(ii)$, $(iii)$ follows from the identities $^ab = [a, b]b$, $ab = [a, b] ba$ respectively.
\end{proof}

The following lemma might be standard, we include it here for the sake of completeness.

\begin{lemma}\label{L:2.8}
Let $\G$ be a group of nilpotency class $\cc >1$, $r$ be a positive integer and $[ ,\dots, ]$ be a commutator of weight $r+1$. If $a\in \gamma_{n_1}(\G)$, $b\in \gamma_{n_2}(\G)$ and $g_i \in \gamma_{m_i} (\G)$ for $1\le i\le r$, satisfies $n_1 + n_2 +m_1 +\cdots +m_r\ge \cc+1$, then the following holds:
\begin{itemize}
\item[$(i)$] $[ab, g_r, \dots, g_1] = [a, g_r, \dots, g_1] [b, g_r, \dots, g_1]$.
\item[$(ii)$] $[g_r, \dots, g_1, ab] = [g_r, \dots, g_1, a] [g_r, \dots, g_1, b]$.
\item[$(iii)$] $[g_r, \dots g_{i+1}, ab, g_i, \dots, g_1] = [g_r, \dots, g_{i+1}, a, g_i, \dots, g_1] [g_r, \dots, g_{i+1}, b, g_i, \dots, g_1]$ for $1\le i < r$.
\end{itemize} 
\end{lemma}

\begin{proof}We prove $(i)$, $(ii)$ and $(iii)$ by induction on $r$.
\begin{itemize}
\item[$(i)$] Let $r = 1$, then \eqref{eq:2.7.1} yields $[ab, g_1] = [a, b, g_1] [b, g_1] [a, g_1] = [a, g_1] [b, g_1]$. Let $r >1$, and $[ab, g_r, \dots, g_1]$ be of the form $[[ab, g_r], g_{r-1}, \dots, g_1]$. Then applying \eqref{eq:2.7.1} to $[ab, g_r]$ in $[[ab, g_r], g_{r-1}, \dots, g_1]$, and using induction hypothesis we obtain 
\begin{align*} &[[ab, g_r], g_{r-1}, \dots, g_1]=[[a, b, g_r] [b, g_r] [a, g_r], g_{r-1}, \dots, g_1] \\
&= [[a, b, g_r], g_{r-1}, \dots, g_1]\ [[b, g_r], g_{r-1}, \dots, g_1]\ [[a, g_r], g_{r-1}, \dots, g_1]. 
\end{align*}
The first term vanishes and the remaining terms commute by \eqref{eq:2.7.4}, yielding the required result. In case $g_r$ is not next to a right bracket, then observe that there exist a $j\in \{1, \dots, r-1\}$ such that $[ab, g_r, \dots, g_j]$ in $[[ab, g_r, \dots, g_j], \dots g_1]$ can be written as $[ab, h_t, \dots, h_1]$, $h_k\in \G$, $t < r-(j-1)$, giving the result we seek by induction hypothesis. 
\item[$(ii)$] Now using \eqref{eq:2.7.2}, the proof follows mutatis mutandis the proof of $(i)$.
\item[$(iii)$] Note that $[g_2, ab, g_1]$ can be bracketed only in 2 ways, either $[g_2, [ab, g_1]]$ or $[[g_2, ab], g_1]$. Expanding $[ab, g_1]$ in $[g_2, [ab, g_1]]$ by \eqref{eq:2.7.1}, then using $(i)$ we obtain
\begin{align*} &[g_2, [ab, g_1]] = [g_2, [a, b, g_1] [b, g_1] [a, g_1]] \\
& = [g_2, [a, b, g_1]] [g_2, [b, g_1]] [g_2, [a, g_1]].
\end{align*}
The first term vanishes and $[g_2, [b, g_1]], [g_2, [a, g_1]]$ commute by \eqref{eq:2.7.4} yielding $[g_2, [ab, g_1]] = [g_2, [a, g_1]] [g_2, [b, g_1]]$. Similarly we get 
\begin{align*} &[[g_2, ab], g_1] = [[g_2, a] [a, g_2, b] [g_2, b], g_1]\ (\text{by}\ \eqref{eq:2.7.2}) \\
& = [[g_2, a], g_1] [[a, g_2, b], g_1] [[g_2, b], g_1]\ (\text{by}\ (ii)).
\end{align*} 
 The second term vanishes yielding $[[g_2, ab], g_1] = [[g_2, a], g_1] [[g_2, b], g_1]$. Let $r>2$, we proceed by considering the following cases:
 \begin{itemize}
\item[$(a)$] Let a left bracket be next to $ab$, and $[g_r, \dots g_{i+1}, ab, g_i, \dots, g_1]$ be of the form $[g_r, \dots, g_{i+1}, [ab, g_i], \dots, g_1]$. Then expanding $[ab, g_i]$ in $[g_r, \dots, g_{i+1}, [ab, g_i], \dots, g_1]$ by \eqref{eq:2.7.1}, we obtain $[g_r, \dots g_{i+1}, [ab, g_i], \dots, g_1] = [g_r, \dots g_{i+1}, [a, b, g_i] [b, g_i] [a, g_i], \dots, g_1]$. If $i>1$, then induction hypothesis yields
 \begin{align*} &[g_r, \dots g_{i+1}, [a, b, g_i] [b, g_i] [a, g_i], \dots, g_1] = [g_r, \dots g_{i+1}, [a, b, g_i], \dots, g_1] \\& [g_r, \dots g_{i+1}, [b, g_i], \dots, g_1] [g_r, \dots g_{i+1}, [a, g_i], \dots, g_1].
 \end{align*}
 Moreover if $i=1$, we obtain the same expression using $(ii)$.
Now the first term vanishes, and the other terms commute by \eqref{eq:2.7.4} giving the result. Suppose $[g_r, \dots g_{i+1}, ab, g_i, \dots, g_1]$ is not of the form $[g_r, \dots, g_{i+1}, [ab, g_i], \dots, g_1]$, then there exist $j\in \{1, \dots i-1\}$ such that $[ab, g_i, \dots, g_j]$ in $ [g_r, \dots, g_{i+1}, [ab, g_i, \dots, g_j], \dots g_1]$ can be written as $[ab, h_t, \dots, h_1]$, $h_k\in \G$, $t < i-(j-1)$, and the result follows by induction hypothesis.
\item[$(b)$] Let a right bracket be next to $ab$, and $[g_r, \dots g_{i+1}, ab, g_i, \dots, g_1]$ be of the form $[g_r, \dots, [g_{i+1}, ab], g_i, \dots, g_1]$. Then expanding $[g_{i+1}, ab]$ in $[g_r, \dots, [g_{i+1}, ab], g_i, \dots, g_1]$ by \eqref{eq:2.7.2}, we have $[g_r, \dots, [g_{i+1}, ab], g_i, \dots, g_1] = [g_r, \dots, [g_{i+1}, a] [a, g_{i+1}, b] [g_{i+1}, b], g_i, \dots, g_1]$.  If $i = r-1$, then using $(i)$ and if $i<r-1$, then by induction hypothesis, we obtain
 \begin{align*} &[g_r, \dots, [g_{i+1}, a] [a, g_{i+1}, b] [g_{i+1}, b], g_i, \dots, g_1] = [g_r, \dots, [g_{i+1}, a], g_i, \dots, g_1]\\& [g_r, \dots, [a, g_{i+1}, b], g_i, \dots, g_1]\ [g_r, \dots, [g_{i+1}, b], g_i, \dots, g_1].
 \end{align*}
The second term vanishes giving us the required result. Otherwise there exist $j\in \{i+2, \dots, r\}$ such that $[g_j,\dots g_{i+1}, ab]$ in $[g_r, \dots, [g_j,\dots g_{i+1}, ab], g_i, \dots, g_1]$ can be written as $[h_t, \dots, h_1, ab]$, $h_k\in \G$, $t <j- i$, which yields the result by induction hypothesis.
\end{itemize}
\end{itemize}
\end{proof}

\begin{corollary}\label{cr:2.9}
Let $\G$ be a group of nilpotency class $\cc >1$, and $[ , \dots, ]$ be a commutator of weight $\cc$. Then the map $[ , \dots, ]\colon \underbrace{\G \times \G \times \cdots \times \G}_{\cc \, times}$ $\longrightarrow \gamma_{\cc} (\G)$ given by $(g_c, \dots, g_1)\mapsto [g_c, \dots, g_1]$ is multiplicative in each coordinate.
\end{corollary}

\begin{lemma}\label{L:2.10}
 Let $\G$ be a nilpotent group of class $\cc$ and $a\in \gamma_i(\G)$, $b\in \gamma_j(\G)$.
 \begin{itemize}
 \item[$(i)$] If $2i+3j\ge \cc+1$, then $[b^n, a] = \prod_{t = n}^{2} [_tb, a]^{{{n}\choose{t}}} [b, a]^n$ for all $n\in \mathbb{N}$. Moreover, if\; $i+rj\ge \cc+1$ and $2i+3j\ge \cc+1$, then $[b^n, a] = \prod_{t = r-1}^{2} [_tb, a]^{{{n}\choose{t}}} [b, a]^n$ for all $n\in \mathbb{N}$.
 \item[$(ii)$] If $3i+2j\ge \cc+1$, then $[b, a^n] = \prod_{t = n-1}^{1} [_ta, b, a]^{{{n}\choose{t+1}}} [b, a]^n$ for all $n\in \mathbb{N}$. Moreover, if $ri+j\ge \cc+1$ and $3i+2j\ge \cc+1$, then $[b, a^n] = \prod_{t = r-2}^{1} [_ta, b, a]^{{{n}\choose{t+1}}} [b, a]^n$ for all $n\in \mathbb{N}$.
 \item[$(iii)$] If $i+3j\ge \cc+1$ and the order of $b$ is $n$ modulo the center, then the order of $[b, a]$ is at most $n$ if $n$ is odd, and at most $2n$ if $n$ is even.
 \item[$(iv)$] If $3i+j\ge \cc+1$ and the order of $a$ is $n$ modulo the center, then the order of $[b, a]$ is at most $n$ if $n$ is odd, and at most $2n$ if $n$ is even.
 \end{itemize}
 \end{lemma}
 \begin{proof}
 The first two parts can be proved using the commutator identities in Lemma \ref{L:2.7} and by induction on $n$. To prove $(iii)$ we set $m = n$ in case $n$ is odd and $m=2n$ in case $n$ is even. Using $(i)$, we have $1=[b^m, a] = [b, b, a]^{{{m}\choose{2}}} [b, a]^m$. Applying Lemma \ref{L:2.8} yields $ [b^{{{m}\choose{2}}}, b, a]=[b, b, a]^{{{m}\choose{2}}}$. Since $n\mid {{m}\choose{2}}$, we obtain $[b^{{{m}\choose{2}}}, b, a] = 1$, proving $[b, a]^m =1$. Similarly $(iv)$ can be proven using $(ii)$.
\end{proof}

In the next Theorem, we recall the collection formula given in Theorem $12.3.1$ of \cite{MH}.
\begin{theorem} \label{th:2.11}
We may collect the product $(x_1x_2 \dots x_j)^n$ in the form $(x_1x_2 \dots x_j)^n =x_1^n x_2^n\dots x_j^n C_{j+1}^{e_{j+1}}\dots C_k^{e_k}R_1R_2\dots R_s$, where $C_{j+1}, \dots C_k$ are the basic commutators on $x_1, x_2, \dots x_j$ in order, and $R_1, \dots, R_s$ are basic commutators later than $C_i$ in the ordering. For $j+1\le i\le k$, the exponent $e_i$ is of the form $e_i = a_1n + a_2 n^{(2)} + \cdots +a_m n^{(m)}$, where $m$ is the weight of $C_i$, the $a'$s are non-negative integers and do not depend on $n$ but only on $C_j$. Here $n^{(l)} = {{n}\choose{l}}$. 
\end{theorem}

Let $\mathcal{A}(\{a, b\})$ denote the set of all basic commutators of $a, b$. Set $\mathcal{A}_i(\{a, b\}) = \{C\in \mathcal{A}(\{a, b\}) : w(C)\le i\}$, in which $w(C)$ denote the weight of $C$ in $a, b$. We restate Theorem \ref{th:2.11} in the following way.

\begin{lemma}\label{L:2.12}
Let $\G$ be a group, $a, b\in \G$ and $n, i\in \mathbb{N}$. Then $(ab)^n = \prod_{C\in \mathcal{A}_i(\{a, b\})} C^{f_C(n)} a^n b^n\mod \gamma_{i+1} (\langle a, b\rangle)$. For $C\in \mathcal{A}(\{a, b\})$, $f_C(n) = a_1{{n}\choose{1}} + a_2{{n}\choose{2}} +\dots +a_{r_C}{{n}\choose{r_C}}$, where $r_C$ is the largest $r$ with $a_r \ne 0$ and $a_1, a_2, \dots, a_{r_C}$ are non-negative integers depending only on $C$.
\end{lemma}
The following remark can be found in Lemma $3.2.5$ of \cite{SM}.

\begin{remark}\label{R:2:13}
When applying the commutator collection process to $(ab)^n$, if $C_i = [b,\  _ta]$, then $e_i = {{n}\choose{t+1}}$.
\end{remark}

\begin{remark}\label{R:2.14}
Since we follow left notations, $[b,\  _ta]$ in Remark \ref{R:2:13} corresponds to the commutator $[\ _tb, a]$.
\end{remark}

The next lemma can be found in \cite{SM}.

\begin{lemma}\label{L:2.15} 
Let $\G$ be a regular $\p$-group and $n$ be a positive integer. Then the following holds:
\begin{itemize}
\item[$(i)$] For all $a, b\in \G$, the following are equivalent: $[b, a]^{\p^n} =1; [b, a^{\p^n}] =1; [b^{\p^n}, a] =1$.
\item[$(ii)$] For all $a\in \G$, any commutator $C$ of weight at least $1$ in $a$ has order at most the order of $a$ modulo the center.
\item[$(iii)$] For all $g_1, g_2, \dots, g_r\in \G$, the order of the product $g_1g_2\dots g_r$ is at most the maximum of the orders of $g_1, g_2, \dots, g_r$.
\item[$(iv)$] For all $a, b\in \G$, $a^{\p^n} =b^{\p^n}$ if and only if $(ab^{-1})^{\p^n} =1$.
\end{itemize}
\end{lemma}

\section{On a Theorem of Schur and solving the Conjecture for $\p$-groups of class at most $\p$}
For an odd prime $\p$, the conjecture was proved for $\p$-groups of class at most $\p-2$ by the author of \cite{PM3} and for $\p$-groups of class at most $\p-1$ by authors of \cite{MHM1}. We prove this result for $\p$-groups of class at most $\p$, by showing that the exponent of the commutator subgroup of a Schur cover of $\G$ divides exponent of $\G$. We refer the reader to \cite{KG} for an account on central extensions and Schur covers. In this section we also prove the generalization of the theorem by I. Schur (\cite{IS}) mentioned in the introduction.

A group $\G$ is said to be $n$-central if exponent of $G/{\Z}$ divides $n$ and $\G$ is said to be $n$-abelian if for all $a, b\in \G$, $(ab)^n = a^n b^n$. In \cite{PM1}, the author has proved that regular $\p$-groups have zero exponential rank. Thus for a regular $\p$-group $\G$, $\p^n$-central implies $\p^n$-abelian, which can also be seen using Lemma \ref{L:3.1}. Moreover Lemma \ref{L:3.2} for the case $n=1$ was proved by A. Mann in Lemma $9$ of \cite{AM}. The next two lemmas follow from Lemma \ref{L:2.15}. We include them for the sake of completeness.
 
\begin{lemma}\label{L:3.1}
Let $\G$ be a regular $\p$-group and $a, b\in \G$. If $[b, a^{\p^n}] =1$, then $(ab)^{\p^n} =a^{\p^n}b^{\p^n}$ for all $n\in \mathbb{N}$.
\end{lemma}
\begin{proof} Since $\G$ is regular $(ab)^{\p^n} = a^{\p^n}b^{\p^n}s^{\p^n}$, where $s\in \gamma_2\langle a, b\rangle$. Using Lemma \ref{L:2.15}$(i)$, we get $[b, a]^{\p^n} =1$. Note that $\gamma_2\langle a, b\rangle = \langle \{^g [b, a] : g\in \langle a, b\rangle \}\rangle$. Hence by Lemma \ref{L:2.15}$(iii)$, we obtain $s^{\p^n} =1$ completing the proof.
\end{proof}

\begin{lemma}\label{L:3.2}
Let $\G$ be a $\p$-group of nilpotency class $\p$ and $a, b\in \G$. Then the following are equivalent : $[b, a]^{\p^n} =1; [b, a^{\p^n}] =1; [b^{\p^n}, a] =1$ for all $n\in \mathbb{N}$.
\end{lemma}
\begin{proof}
By symmetry it is enough to show $[b, a]^{\p^n} =1$ if and only if $[b, a^{\p^n}] =1$. We have $[b, a^{\p^n}] = ba^{\p^n}b^{-1}(a^{-1})^{p^n} = (^b a)^{\p^n}(a^{-1})^{\p^n}$ and $[b, a]^{\p^n} = (^b a a^{-1})^{\p^n}$. Note that the group $\langle ^b a, a \rangle =\langle [b, a], a \rangle$ has nilpotency class $\le \p-1$. Hence using Lemma \ref{L:2.15}$(iv)$, we obtain $(^b a)^{\p^n} =(a)^{\p^n}$ if and only if $(^b a a^{-1})^{\p^n} =1$. Thus $[b, a^{\p^n}] =1$ if and only if $[b, a]^{\p^n} =1$.
\end{proof}

\begin{corollary}\label{cr:3.3}
Let $\G$ be a $\p$-group of nilpotency class $\p$. Then $\e(\gamma_2(\G))\mid \e(\G/{\Z})$.
\end{corollary}
Corollary \ref{cr:3.3} follows from Lemma \ref{L:3.2} and Lemma \ref{L:3.1}. 

\begin{corollary}\label{cr:3.4}
Let $\G$ be a $\p$-group of nilpotency class at most $\p-1$. Then $\e(\M)\mid \e(\G)$.
\end{corollary}

\begin{proof}
Applying Corollary \ref{cr:3.3} to a Schur cover of $\G$ yields the proof.
\end{proof}

The next lemma is a generalization of Lemma \ref{L:2.15}$(ii)$ for groups of class $\p$.
\begin{lemma}\label{L:3.5}
Let $\G$ be a $\p$-group of nilpotency class $\p$. Then for all $a\in \G$, any commutator $C$ of weight at least $1$ in $a$ has order at most the order of $a$ modulo the center.
\end{lemma}
\begin{proof}
Let the order of $a$ modulo the center be $\p^n$, $n\ge 1$. If $C= [g_1, a]$ or $C = [a, g_2]$, $g_1, g_2\in G$, then applying Lemma \ref{L:3.2} yields $C^{\p^n} = 1$. We proceed by induction on the weight of $C$. Let $w(C) >2$ and $C=[C_2, C_1]$, where both $C_1, C_2$ are having lesser weight than $C$. If either $C_1$ or $C_2$ is $a$, then using Lemma \ref{L:3.2}, we obtain $C^{\p^n} = 1$. Otherwise, suppose $C_1$ has weight $\ge 1$ in $a$. Then by induction hypothesis, $C_1^{\p^n} = 1$ and thus $[C_2, C_1^{\p^n}] =1$. Hence Lemma \ref{L:3.2} implies that $C^{\p^n}=1$.
\end{proof}

Observe that if $\G$ is a group of nilpotency class $3$, then for every $a, b\in \G$ $[b, a^n] = [a, b, a]^{{n}\choose{2}} [b, a]^n$ for all $n\in \mathbb{N}$. So Lemma \ref{L:3.6} clearly holds for the case $\p = 2$. 

\begin{lemma}\label{L:3.6}
Let $\p$ be an odd prime and $n$ be a positive integer. Suppose $\G$ is a $\p$-group and $\HH=\langle a, b \rangle$ has nilpotency class at most $\p+1$ for some $a,b\in \G$. If $[b, a^{\p^n}] = 1$, then $[\ _{\p-1}a, [h, a]]^{{\p^n}\choose{\p}}[h, a]^{\p^n}=1$ for all $h\in \HH$.
\end{lemma}

\begin{proof}
Since $\HH = \langle a, b \rangle$, $a^{\p^n}\in \ZZ(\HH)$. So $1= [h, a^{\p^n}] = ha^{\p^n}h^{-1}a^{-\p^n} = (^h a)^{\p^n}a^{-\p^n} = ([h, a]a)^{\p^n}a^{-\p^n}$. Note that the group $\langle [h, a], a\rangle$ has class $\le \p$. Applying Lemma \ref{L:2.12} to $([h, a]a)^{\p^n}$, we have
 \begin{equation} \label{eq:3.6.1} 1 = \prod_{C\in \mathcal{A}_{\p}(\{[h, a], a\})} C^{f_C(\p^n)} [h, a]^{\p^n}.
\end{equation}
Note that Lemma \ref{L:3.5} implies that $C^{\p^n} = 1$ for every $C\in \mathcal{A}(\{[h, a], a\})$. Consider a $C\in \mathcal{A}_{\p}(\{[h, a], a\})$ and let $n_1$, $n_2$ be weights of $C$ in $[h, a]$, $a$ respectively. Then $n_1 + n_2\le \p$ and since $C\in \gamma_{2n_1+n_2}(\HH)$, $2n_1+n_2\le \p+1$. From Lemma \ref{L:2.12}, recall that  $r_C\le n_1 + n_2$. So if $n_1 + n_2<\p$, then $r_C< \p$. Moreover if $n_1 + n_2 = \p$, then $2n_1 + n_2 \le \p+1$ yielding $n_1 = 1$ and $n_2 = \p-1$. Therefore $r_C < \p$ for every $C\in \mathcal{A}_{\p} (\{[h, a], a\})$, hence $\p^n\mid f_C(\p^n)$, except for $C = [\ _{\p-1}a, [h, a]]$. Also $f_{[\ _{\p-1}a, [h, a]]} (\p^n) = {{\p^n}\choose{\p}}$ by Remark \ref{R:2.14}. Thus \eqref{eq:3.6.1} gives the result we seek.
\end{proof}

\begin{definition} \label{D:3.7}
Define $\alpha_m (n)=\sum_{1\le i_1 <i_2 <\dots<i_{m-1} <n} {{n}\choose{i_{m-1}}} {{i_{m-1}}\choose{i_{m-2}}} \dots {{i_2}\choose{i_1}}$, 
where $1< m\le n\in \mathbb{N}$.
\end{definition}

\begin{lemma} \label{L:3.8}
For $2 < m\le n$, $\alpha_{m} (n) = \sum_{k=m-1}^{n-1} {{n}\choose{k}} \alpha_{m-1} (k)$.
\end{lemma}
\begin{proof}
Evaluating the sum on the right
\begin{align*} \sum_{k=m-1}^{n-1} {{n}\choose{k}} \alpha_{m-1}(k) &=\sum_{k=m-1}^{n-1} \Bigg({{n}\choose{k}}\times \sum_{1\le i_1 <i_2 <\dots<i_{m-2} <k} {{k}\choose{i_{m-2}}} {{i_{m-2}}\choose{i_{m-3}}} \dots {{i_2}\choose{i_1}}\Bigg)\\
&=\sum_{1\le i_1 <i_2 <\dots<i_{m-2} <k <n} {{n}\choose{k}} {{k}\choose{i_{m-2}}} \dots {{i_2}\choose{i_1}} \\
&= \alpha_m(n).
\end{align*}
\end{proof}

If $\G$ is a $2$-group of class $3$ and $\e(\G/{\Z}) = 2^n$, then $[b, a^{2^n}]=[a, b, a]^{{2^n}\choose{2}} [b, a]^{2^n}=1$ for all $a, b\in \G$. Since $a, b\in \G$ are arbitrary, replacing $a$ with $ab$ in $[a, b, a]^{{2^n}\choose{2}} [b, a]^{2^n}=1$ yields $[ab, b, ab]^{{2^n}\choose{2}} [a, ab]^{2n} = [a, b, a]^{{2^n}\choose{2}} [b, b, a]^{{2^n}\choose{2}} [b, a]^{2^n}=1$. Thus we obtain $[b, b, a]^{{2^n}\choose{2}} =1$. Moreover $[b, b, a] = [b, a, b]^{-1}$, yielding $[b, a, b]^{{2^n}\choose{2}}=1$. Now interchanging the roles of $a$ and $b$ gives $[a, b, a]^{{2^n}\choose{2}} =1$. Thus we obtain $[b, a]^{2n} = 1$, and hence $\e(\gamma_2(\G))\mid 2n$. In Theorem \ref{th:3.9}, we ignore the case $\p = 2$ since the conjecture was proved for groups of class $2$ in \cite{MRJ}.

\begin{theorem}\label{th:3.9}
Let $\p$ be an odd prime and $\G$ be a $\p$-group of nilpotency class $\p+1$. If $\G$ is $\p^n$-central, then $\e( \gamma_2(\G)) \mid \p^n$.
\end{theorem}
\begin{proof} Note that $\gamma_{\p} (\gamma_2(\G))\subset \gamma_{2\p} (\G)\subset \gamma_{\p+2}(\G) = 1$. Thus $\gamma_2(\G)$ is $\p^n$-abelian by Lemma \ref{L:3.1}. Hence it is enough to show $[b, a]^{\p^n} =1$ for every $a, b\in \G$. Let $a, b\in \G$, applying Lemma \ref{L:3.6} we have
\begin{equation}\label{eq:3.9.1} {[\ _{\p-1}a, [b, a]]}^{{\p^n}\choose{\p}} {[b, a]}^{\p^n}=1.
\end{equation}
In the remaining proof we proceed to show ${[\ _{\p-1}a, [b, a]]}^{{\p^n}\choose{\p}}=1$. Towards that end, since $a, b\in \G$ are arbitrary, we keep replacing $a$ with $ab$ starting from \eqref{eq:3.9.1}. Let $S=\{ [x_{\p-1}, x_{\p-2}, \dots, x_1, [b, a]] : x_1, x_2, \dots, x_{\p-1}\in \{a, b\} \}$, where the commutators in $S$ are right normed. For $1\le i\le \p-1$, define $\sigma_i \colon S \longrightarrow \{a, b\}$ as $\sigma_i([x_{\p-1}, x_{\p-2}, \dots, x_1, [b, a]]) = x_i$.
 For $C \in S$, define $S_C = \{ D\in S :  \forall \ 1\le i\le \p-1,\ \mbox{if}\ \sigma_i(C)=b,\ \mbox{then}\ \sigma_i(D)=b\}$.
Observe that all the elements of $S$ commute with one another. Then using Corollary \ref{cr:2.9}, replacing $a$ in $C$ with $ab$ gives $\prod_{D\in S_C} D$.
For $0\le r\le \p-1$, let $S_r = \{ C\in S : \sigma_i (C)=b\, \mbox{for\ exactly}\ r\  \sigma_i,\ 1\le i\le \p-1\}$ and $\E_r \ =\prod_{C\in S_r}C$. Consider $C\in S_r$ and $D\in S_t$. Note that $D\in S_C$ implies $t\ge r$. Moreover for every $D\in S_t$, the number of $C\in S_r$ having $D\in S_C$ is ${{t}\choose{r}}$. Thus replacing $a$  with $ab$ in $\E_r$ gives $\prod_{C\in S_r}(\prod_{D\in S_C}D) =\prod_{t=r}^{\p-1}\E_t^{{{t}\choose{r}}}$. Now we replace $a$ in \eqref{eq:3.9.1} with $ab$ and observe that $[\ _{\p-1}a, [b, a]]=\E_0$ and $[b, ab] = [b, a]$. After this replacement \eqref{eq:3.9.1} becomes
\begin{equation}\label{eq:3.9.2}  
(\prod_{r=0}^{p-1}{\E_r})^{{\p^n}\choose{\p}} {[b, a]}^{\p^n}=1.
\end{equation}

Comparing \eqref{eq:3.9.1} with \eqref{eq:3.9.2} yields 
\begin{equation}\label{eq:3.9.3} \prod_{r=1}^{\p-1}{\E_r}^{{\p^n}\choose{\p}}=1.
\end{equation}
By replacing $a$ in \eqref{eq:3.9.3} with $ab$, we obtain 
\begin{equation}\label{eq:3.9.4} \E_{\p-1}^{{\p^n}\choose{\p}}\prod_{r=1}^{\p-2}\big(\prod_{t=r}^{\p-1}{\E_t}^{{t}\choose{r}}\big)^{{\p^n}\choose{\p}}=1. 
\end{equation}
By comparing \eqref{eq:3.9.3} and \eqref{eq:3.9.4}, we obtain $\prod_{r=1}^{\p-2}\big(\prod_{t=r+1}^{\p-1}{\E_t}^{{t}\choose{r}}\big)^{{\p^n}\choose{\p}}=1$. Now rearranging $\prod_{r=1}^{\p-2}\big(\prod_{t=r+1}^{\p-1}{\E_t}^{{t}\choose{r}}\big)$ as $\prod_{t=2}^{\p-1}\big(\prod_{r=1}^{t-1}{\E_t}^{{t}\choose{r}}\big)$ gives
\begin{equation}\label{eq:3.9.5} \prod_{t=2}^{\p-1}({\E_t}^{\sum_{r=1}^{t-1} {{t}\choose{r}}})^{{\p^n}\choose{\p}}=1.
\end{equation}
We claim that for $2\le m\le \p-1$,
\begin{equation}\label{eq:3.9.6} \prod_{k=m}^{\p-1}\big({\E_k}^{\alpha_m(k)}\big)^{{\p^n}\choose{\p}}=1. 
\end{equation}
Observe that $\sum_{r=1}^{t-1} {{t}\choose{r}} = \alpha_2(t)$, and hence \eqref{eq:3.9.5} becomes $\prod_{t=2}^{\p-1}\big({\E_t}^{\alpha_2(t)}\big)^{{\p^n}\choose{\p}}=1$. Now we proceed by induction on $m$. For $2 < m\le \p-1$, assuming \eqref{eq:3.9.6} for $m-1$ yields

\begin{equation}\label{eq:3.9.7}
 \prod_{k=m-1}^{\p-1}\big({\E_k}^{\alpha_{m-1}(k)}\big)^{{\p^n}\choose{\p}}=1.
\end{equation}

By replacing $a$ with $ab$ in \eqref{eq:3.9.7}, we get

\begin{equation}\label{eq:3.9.8} 
\big(\E_{\p-1}^{\alpha_{m-1}(\p-1)}\big)^{{\p^n}\choose{\p}}\prod_{k=m-1}^{\p-2}\big(\prod_{j=k}^{\p-1}{\E_j}^{{j}\choose{k}}\big)^{\alpha_{m-1}(k){{\p^n}\choose{\p}}}=1. 
\end{equation}

Now comparison of \eqref{eq:3.9.7} and \eqref{eq:3.9.8} gives $\prod_{k=m-1}^{\p-2}\big(\prod_{j=k+1}^{\p-1}{\E_j}^{{j}\choose{k}}\big)^{\alpha_{m-1}(k) {{\p^n}\choose{\p}}}=1$. Rearranging  $\prod_{k=m-1}^{\p-2}\big(\prod_{j=k+1}^{\p-1}{\E_j}^{{j}\choose{k}}\big)^{\alpha_{m-1}(k)}$ as $\prod_{j=m}^{\p-1}\big(\prod_{k=m-1}^{j-1}{\E_j}^{{{j}\choose{k}} \alpha_{m-1}(k)}\big)$ and using Lemma \ref{L:3.8}, we obtain

\begin{equation}\label{eq:3.9.9}
\prod_{j=m}^{\p-1}\big(\E_j^{\alpha_m(j)}\big)^{{\p^n}\choose{\p}}=1.
\end{equation}
This proves the claim. Setting $m=p-1$ in \eqref{eq:3.9.6} yields $\E_{\p-1}^{\alpha_{\p-1} (\p-1){{\p^n}\choose{\p}}}=1$. Note that $\alpha_{\p-1} (\p-1)=(\p-1)!$ is relatively prime to $\p$, so
 
\begin{equation} \label{eq:3.9.10} \E_{\p-1}^{{\p^n}\choose{\p}}=1.
\end{equation}
 
We have $\E_{\p-1}=[\ _{\p-1}b, [b, a]]$, and interchanging $a$ and $b$ in \eqref{eq:3.9.10} gives $[\ _{\p-1}a, [a, b]]^{{\p^n}\choose{\p}} =1$. Since $[\ _{\p-1}a, [b, a]]={[\ _{\p-1}a, [a, b]]}^{-1}$, we obtain $[\ _{\p-1}a, [b, a]]^{{\p^n}\choose{\p}}=1$, proving the theorem.
\end{proof}

Keeping the notations as in the proof of Theorem \ref{th:3.9}, we give an illustration of Theorem \ref{th:3.9} by taking the special case $\p = 5$. 

\begin{example} \label{E:3.10}
\begin{align*} \E_0 =& [\ _4a, [b, a]], \\ \E_1 =& [b, a, a, a, [b, a]] [a, b, a, a, [b, a]] [a, a, b, a, [b, a]] [a, a, a, b, [b, a]],\\ \E_2 = &[b, b, a, a, [b, a]] [b, a, b, a, [b, a]] [a, b, b, a, [b, a]] \\  & [b, a, a, b, [b, a]] [a, b, a, b, [b, a]] [a, a, b, b, [b, a]], \\ \E_3 =& [b, a, a, a, [b, a]] [a, b, a, a, [b, a]] [a, a, b, a, [b, a]] [a, a, a, b, [b, a]], \\ \E_4 = &[\ _4b, [b, a]].
\end{align*}
Replacing $a$ in $\E_0$ with $ab$ gives $\E_0\E_1\E_2\E_3\E_4$. Similarly replacing $a$ in $\E_1, \E_2, \E_3, \E_4$ with $ab$ gives $\E_1\E_2^2\E_3^3\E_4^4$, $\E_2\E_3^3\E_4^6$, $\E_3\E_4^4$, $\E_4$ respectively. By \eqref{eq:3.9.1} we have
  \begin{equation} \label{eq:3.10.1} \E_0^{{{5^n}\choose{5}}} [b, a]^{5^n} = 1. 
 \end{equation} 
 
Replacement of $a$ in \eqref{eq:3.10.1} with $ab$ gives 
\begin{equation} \label{eq:3.10.2}(\E_0\E_1\E_2\E_3\E_4)^{{{5^n}\choose{5}}} [b, a]^{5^n}= 1.
\end{equation}

 Comparing \eqref{eq:3.10.1} and \eqref{eq:3.10.2} yields 
 \begin{equation} \label{eq:3.10.3} (\E_1\E_2\E_3\E_4)^{{{5^n}\choose{5}}} = 1.
 \end{equation} 
 Again replacement of $a$ in \eqref{eq:3.10.3} with $ab$ gives 
 \begin{equation} \label{eq:3.10.4} \big( (\E_1\E_2^2\E_3^3\E_4^4) (\E_2\E_3^3\E_4^6) (\E_3\E_4^4) (\E_4) \big)^{{{5^n}\choose{5}}} = 1.
 \end{equation}
 
 By comparing \eqref{eq:3.10.3} and \eqref{eq:3.10.4}, we obtain $\big( (\E_2^2\E_3^3\E_4^4) (\E_3^3\E_4^6) (\E_4^4) \big)^{{{5^n}\choose{5}}} = 1$. Thus
 \begin{equation} \label{eq:3.10.5}
 \big(\E_2^2\E_3^6\E_4^{14}\big)^{{{5^n}\choose{5}}} = 1.
 \end{equation} 
 
 Again replacement of $a$ in \eqref{eq:3.10.5} with $ab$ yields, 
 \begin{equation}\label{eq:3.10.6}
 \big((\E_2\E_3^3\E_4^6)^2 (\E_3\E_4^4)^6 \E_4^{14}\big)^{{{5^n}\choose{5}}} = 1.
 \end{equation}
 
 Now by comparing \eqref{eq:3.10.5} and \eqref{eq:3.10.6} we get $\big( (\E_3^3\E_4^6)^2 (\E_4^4)^6 \big)^{{{5^n}\choose{5}}} = 1$. This gives 
 \begin{equation}\label{eq:3.10.7}
 \big(\E_3^6\E_4^{36}\big)^{{{5^n}\choose{5}}} = 1.
 \end{equation}
 
 After replacing $a$ in \eqref{eq:3.10.7} with $ab$, \eqref{eq:3.10.7} becomes
 \begin{equation} \label{eq:3.10.8}
 \big( (\E_3\E_4^4)^6 \E_4^{36}\big)^{{{5^n}\choose{5}}}=1.
 \end{equation} 
 
 Finally comparing \eqref{eq:3.10.7} and \eqref{eq:3.10.8} yields
 \begin{equation} \label{eq:3.10.9}\big(\E_4^{24}\big)^{{{5^n}\choose{5}}} = 1.
 \end{equation}
 
The values of $\alpha_m(n)$ for $1\le m< n\le4$ are listed below:

\begin{align*}&\alpha_2(2) = {{2}\choose{1}} = 2,\  \alpha_2(3) = {{3}\choose {1}}+{{3}\choose{2}} =6,\ \alpha_2(4) = {{4}\choose{1}}+{{4}\choose{2}}+{{4}\choose{3}} =14,\\ & \alpha_3(3) = {{3}\choose {2}} {{2}\choose {1}}= 6,\ \alpha_3(4) = {{4}\choose {2}}{{2}\choose {1}}+{{4}\choose{3}}{{3}\choose {1}}+{{4}\choose {3}}{{3}\choose {2}} = 36,\\ & \alpha_4(4) = {{4}\choose {3}}{{3}\choose {2}}{{2}\choose {1}} = 24.
\end{align*}

From $\eqref{eq:3.10.5}, \eqref{eq:3.10.7}$ and $\eqref{eq:3.10.9}$, we can see that $\prod_{k=m}^{4}\big({\E_k}^{\alpha_m(k)}\big)^{{5^n}\choose{5}}=1$ for $m = 2, 3, 4$.
\end{example}

\begin{theorem}\label{th:3.11}
Let $\p$ be an odd prime and $\G$ be a finite $\p$-group. If the nilpotency class of $\G$ is at most $\p$, then $\e(\G \wedge \G)\mid \e(\G)$. In particular, $\e(\M)\mid \e(\G)$.
\end{theorem}

\begin{proof}
Let $\HH$ be a Schur cover for $\G$. Since $\HH$ is a central extension of $\G$, the nilpotency class of $\HH$ is at most $\p+1$. Thus by Theorem \ref{th:3.9}, $\e(\gamma_2(\HH))\mid \e(\HH/\ZZ(\HH))\mid \e(\G)$. The theorem now follows by noting that $\G \wedge \G \cong \gamma_2(\HH)$.
\end{proof}

\section{Validity of the Conjecture for $\p$-groups of nilpotency class at most 5 and of odd exponent.}
The author of \cite{PM5} has listed basic commutators in $a, b$ of weight at most $6$ and their respective powers arising in the collection formula. The same collection formula, when we follow left notations is given in $(i)$ of the next lemma.
\begin{lemma}\label{L:4.1}
\begin{itemize}
\item[$(i)$] Let $\G$ be a group of nilpotency class $5$, $a, b\in \G$. Then for all $n\in \mathbb{N}$,
\begin{align*}
(ab)^n =& [[b, a], a, b, a]^{6{{n}\choose{3}}+18{{n}\choose{4}}+12{{n}\choose{5}}} [[b, a], b, b, a]^{{{n}\choose{3}}+7{{n}\choose{4}}+6{{n}\choose{5}}} \\& [a, a, a, b, a]^{3{{n}\choose{4}}+4{{n}\choose{5}}} [a, a, b, b, a]^{{{n}\choose{3}}+6{{n}\choose{4}}+6{{n}\choose{5}}} [a, b, b, b, a]^{3{{n}\choose{4}}+4{{n}\choose{5}}} \\& [b, b, b, b, a]^{{{n}\choose{5}}}  [a, a, b, a]^{2{{n}\choose{3}}+3{{n}\choose{4}}} [a, b, b, a]^{2{{n}\choose{3}}+3{{n}\choose{4}}} \\ & [b, b, b, a]^{{{n}\choose{4}}} [a, b, a]^{{{n}\choose{2}}+2{{n}\choose{3}}} [b, b, a]^{{{n}\choose{3}}} [b, a]^{{{n}\choose{2}}} a^n b^n.
\end{align*}
\item[$(ii)$] Let $\G$ be a group of nilpotency class $6$, $a, b\in \G$. Suppose $a\in \gamma_2(\G)$, then for all $n\in \mathbb{N}$,
\begin{align*}
(ab)^n =& [b, b, b, b, a]^{{{n}\choose{5}}} [a, b, b, a]^{2{{n}\choose{3}}+3{{n}\choose{4}}} [b, b, b, a]^{{{n}\choose{4}}} \\ & [a, b, a]^{{{n}\choose{2}}+2{{n}\choose{3}}} [b, b, a]^{{{n}\choose{3}}} [b, a]^{{{n}\choose{2}}} a^n b^n.
\end{align*}
\item[$(iii)$] Let $\G$ be a group of nilpotency class $6$, $a, b\in \G$. Then for all $n\in \mathbb{N}$, 
\begin{align*}
[b, a^n] =& [a, a, a, a, b, a]^{{{n}\choose{5}}} [[b, a], a, a, b, a]^{2{{n}\choose{3}}+3{{n}\choose{4}}} [a, a, a, b, a]^{{{n}\choose{4}}}\\ & [[b, a], a, b, a]^{{{n}\choose{2}}+2{{n}\choose{3}}}[a, a, b, a]^{{{n}\choose{3}}} [a, b, a]^{{{n}\choose{2}}} [b, a]^n.
\end{align*}
\end{itemize}
\end{lemma}

\begin{proof}
Note that $(ii)$ follows from $(i)$ and for proving $(iii)$, we begin by writing $[b, a^n]$ as $([b, a] a)^n a^{-n}$. Now applying $(ii)$ to $([b, a]a)^n$ yields $(iii)$.
\end{proof}

\begin{lemma}\label{L:4.2}
Let $\G$ be a $3$-group of class $6$ and $a,b\in \G$. If $n>1$ and the order of $a$ modulo the center is $3^n$, then
\begin{itemize}
\item[$(i)$] $[a^3, a^3, a^3, a^3, b, a^3]^{3^{n-2}} =1$.
\item[$(ii)$] $[a^3, a^3, a^3, b, a^3]^{3^{n-2}} =1$.
\item[$(iii)$] $[a^3, a^3, b, a^3]^{3^{n-1}} =1$.
\item[$(iv)$] $[a^3, b, a^3]^{3^{n-1}} =1$.
\item[$(v)$] $[[b, a^3], a^3, b, a^3]^{3^{n-2}} =1$.
\item[$(vi)$] $[a^3, a^3, b, a^3]^{3^{n-2}} =1$.
\item[$(vii)$] $[[b, a^3], a^3, a^3, b, a^3]^{3^{n-2}} =1$.
\end{itemize}
\end{lemma}
\begin{proof}
\begin{itemize}
\item[$(i)$] Applying Lemma \ref{L:2.8} twice yields $[a^3, a^3, a^3, a^3, b, a^3]^{3^{n-2}}$ = $([a^3, a, a^3, a^3, b, a^3]^3)^{3^{n-2}}$ $= [a^{3^n}, a, a^3, a^3, b, a^3] =1$.
\item[$(ii)$] Expanding $[b, a^3]$ in $[a^3, a^3, a^3, b, a^3]$ by Lemma \ref{L:4.1}$(iii)$, and then using Lemma \ref{L:2.8}, we obtain $[a^3, a^3, a^3, b, a^3] = [a^3, a^3, a^3, [a, b, a]]^3 [a^3, a^3, a^3, [b, a]]^3$. Since $a^{3^n}\in \Z$, Lemma \ref{L:2.10}$(iii)$ yields $[a^3, a^3, a^3, [b, a]]^{3^{n-1}}$ = $[a^3, a^3, a^3, [a, b, a]]^{3^{n-1}} =1$. Hence by \eqref{eq:2.7.4}, we have $[a^3, a^3, a^3, b, a^3]^{3^{n-2}} =1$.
\item[$(iii)$] We have $[a^{3^n}, a^3, b, a^3] = 1$. Now expanding $[(a^3)^{3^{n-1}}, a^3, b, a^3]$ by Lemma \ref{L:2.10}$(i)$ yields $[a^3, a^3, a^3, a^3, b, a^3]^{{{3^{n-1}}\choose{3}}} [a^3, a^3, a^3, b, a^3]^{{{3^{n-1}}\choose{2}}} [a^3, a^3, b, a^3]^{3^{n-1}} = 1$. Hence it is clear from $(i)$ and $(ii)$ that $[a^3, a^3, b, a^3]^{3^{n-1}} =1$.
\item[$(iv)$] Since $[a^{3^n}, b, a^3] = 1$, expanding $[(a^3)^{3^{n-1}}, b, a^3]$ by Lemma \ref{L:2.10}$(i)$ gives $[a^3, a^3, a^3, a^3, b, a^3]^{{{3^{n-1}}\choose{4}}} [a^3, a^3, a^3, b, a^3]^{{{3^{n-1}}\choose{3}}} [a^3, a^3, b, a^3]^{{{3^{n-1}}\choose{2}}} [a^3, b, a^3]^{3^{n-1}} = 1$. Now the result follows from $(i)$, $(ii)$ and $(iii)$.
\item[$(v)$] Expanding the left most $[b, a^3]$ in $[[b, a^3], a^3, b, a^3]$ by Lemma \ref{L:4.1}$(iii)$, and then applying Lemma \ref{L:2.8} yields $[[b, a^3], a^3, b, a^3] = [[a, b, a], a^3, b, a^3]^3 [[b, a], a^3, b, a^3]^3$. Since $[a^3, b, a^3]^{3^{n-1}} = 1$, by Lemma \ref{L:2.10}$(iv)$, $[[a, b, a], a^3, b, a^3]$ and $[[b, a], a^3, b, a^3]$ have orders at most $3^{n-1}$. Now \eqref{eq:2.7.4} yields $[[b, a^3], a^3, b, a^3]^{3^{n-2}} =1$.
\item[$(vi)$] Applying Lemma \ref{L:2.10}$(i)$ to $[a^3, a^3, b, a^3]$ gives $[a^3, a^3, b, a^3]$ = $[a, a, a, a^3, b, a^3]$ $[a, a, a^3, b, a^3]^3 [a, a^3, b, a^3]^3$. Since $[a^3, b, a^3]^{3^{n-1}} = 1$, by Lemma \ref{L:2.10}$(iv)$, $[a, a, a^3, b, a^3]$ and $[a, a^3, b, a^3]$ have orders at most $3^{n-1}$. Further using Lemma \ref{L:2.8}, we have $[a, a, a, a^3, b, a^3]^{3^{n-2}}$ $= ([a, a, a, a, b, a]^9)^{3^{n-2}}$ $ = [a^{3^n}, a, a, a, b, a] = 1$. Hence \eqref{eq:2.7.4} yields $[a^3, a^3,b, a^3]^{3^{n-2}} = 1$.
\item[$(vii)$] Since $[a^3, a^3, b, a^3]^{3^{n-2}} = 1$, Lemma \ref{L:2.10}$(iv)$ gives $[[b, a^3], a^3, a^3, b, a^3]^{3^{n-2}} = 1$.
\end{itemize}
\end{proof}

\begin{theorem}\label{th:4.3}
Let $\G$ be a $3$-group of class $6$, $a, b\in \G$ and $c\in \gamma_2(\G)$. 
\begin{itemize}
\item[$(i)$] If $a^{3^n}\in \Z$, then $[b, a^3]^{3^{n-1}} = 1$.
\item[$(ii)$] If $a^{3^n}\in \Z$ and $c^{3^{n-1}} = 1$, then $([b, a^3]c)^{3^{n-1}} = 1$.
\end{itemize}
\end{theorem}
\begin{proof}
\begin{itemize}
\item[$(i)$] Since $a^{3^n}\in\Z$, applying Lemma \ref{L:4.1}$(iii)$ to $[b, (a^3)^{3^{n-1}}]$ yields 
\begin{align*} 1 = & [a^3, a^3, a^3, a^3, b, a^3]^{{{3^{n-1}}\choose{5}}} [[b, a^3], a^3, a^3, b, a^3]^{2{{3^{n-1}}\choose{3}}+3{{3^{n-1}}\choose{4}}}\\ & [a^3, a^3, a^3, b, a^3]^{{{3^{n-1}}\choose{4}}} [[b, a^3], a^3, b, a^3]^{{{3^{n-1}}\choose{2}}+2{{3^{n-1}}\choose{3}}} \\ &[a^3, a^3, b, a^3]^{{{3^{n-1}}\choose{3}}} [a^3, b, a^3]^{{{3^{n-1}}\choose{2}}} [b, a^3]^{3^{n-1}}.
\end{align*}
Now by Lemma \ref{L:4.2}, we obtain that $[b, a^3]^{3^{n-1}} = 1$.
\item[$(ii)$] Expanding $([b, a^3] c)^{3^{n-1}}$ by Lemma \ref{L:4.1}$(ii)$ gives 
\begin{align*} ([b, a^3] c)^{3^{n-1}} =&  [[b, a^3], c, b, a^3]^{{{3^{n-1}}\choose{2}}+2{{3^{n-1}}\choose{3}}} \\ & [c, c, b, a^3]^{{{3^{n-1}}\choose{3}}} [c, b, a^3]^{{{3^{n-1}}\choose{2}}} [b, a^3]^{3^{n-1}} c^{3^{n-1}}.
\end{align*}
Since $[b, a^3]^{3^{n-1}} = 1$, Lemma \ref{L:2.10}$(iv)$ yields $[c, b, a^3]^{3^{n-1}} = 1$. Now by applying Lemma \ref{L:2.8}, we have that $[[b, a^3], c, b, a^3]^{3^{n-2}} = ([[b, a], c, b, a^3]^3)^{3^{n-2}} = [[b, a], c^{3^{n-1}}, b, a^3] = 1$, and similarly we have $[c, c, b, a^3]^{3^{n-2}} = ([c, c, b, a]^3)^{3^{n-2}} = [c^{3^{n-1}}, c, b, a] = 1$. Therefore $([b, a^3]c)^{3^{n-1}} = 1$.
\end{itemize}
\end{proof}
Recall that if $\G$ has nilpotency class $5$, then $\eta(\G)$ will have nilpotency class at most $6$. In the proof of next Lemma we use the isomorphism of $\G \otimes \G$ with the subgroup $[\G, \G^{\phi}]$ of $\eta(\G)$.
\begin{lemma}\label{L:4.4}
Let $\G$ be a $3$-group of class less than or equal to $5$ and exponent $3^n$. Then the exponent of the image of $\G^3 \wedge \G$ in $\G \wedge \G$ divides $3^{n-1}$.
\end{lemma}
\begin{proof}
Let $g, h, a_i, b_i \in \G$ for $i=1, 2$. Consider the isomorphism $\psi : \G \otimes \G \rightarrow [\G, \G^{\phi}]$  defined by $\psi(g \otimes h) = [g, h^{\phi}] $, where $\G^{\phi}$ is an isomorphic copy of $\G$. Applying Theorem \ref{th:4.3}, we obtain that $[a_i^3, b_i^{\phi}]^{3^{n-1}} = 1$ for $i=1, 2$, and $([a_1^3, b_1^{\phi}] [a_2^3, b_2^{\phi}])^{3^{n-1}} = 1$. Thus $(a_i^3\wedge b_i)^{3^{n-1}} = 1$ and $((a_1^3\wedge b_1) (a_2^3\wedge b_2))^{3^{n-1}} = 1$, proving the result we seek.
\end{proof}

The following Lemma can be found in \cite{GE1} which is used for the proof of the main theorem of this section.

\begin{lemma}\label{L:4.5}
Let $\N, \MM$ be normal subgroups of a group $\G$. If $\MM \subset \N$, then we have the exact sequence
$\MM \wedge \G\rightarrow \N\wedge \G \rightarrow \frac{\N}{\MM} \wedge \frac{\G}{\MM} \rightarrow 1.$
\end{lemma}

Now we come to the main theorem of this section.

\begin{theorem}\label{th:4.6}
Let $\G$ be a finite $\p$-group of nilpotency class 5. If $\p$ is odd, then $\e(\G\wedge \G)\mid \e(\G)$. In particular, $\e(\M)\mid \e(\G)$.
\end{theorem}
\begin{proof}
The claim holds when $\p \geq 5$ by Theorem $\ref{th:3.11}$. Now we proceed to prove for $\p = 3$.
Consider the exact sequence $\G^3\wedge \G \rightarrow \G \wedge \G \rightarrow \frac{\G}{\G^3}\wedge \frac{\G}{\G^3} \rightarrow1$. Thus $\e( \G \wedge \G )\mid \e(\im( \G^{3}\wedge \G))\ \e( \frac{\G}{\G^{3}}\wedge \frac{\G}{\G^{3}} )$. Note that $\e(\frac{\G}{\G^{3}}) = 3$ and hence $\e( \frac{\G}{\G^{3}}\wedge \frac{\G}{\G^{3}} ) \mid 3$ (cf. \cite{PM1}). By Lemma \ref{L:4.4}, we have $\e(\im( \G^{3}\wedge \G)) \mid 3^{n-1}$ and the result follows.
\end{proof}

\section{Regular groups, Powerful groups and groups with power-commutator structure}

In this section, we prove that $\e(\M)\mid \e(\G)$ for powerful $\p$-groups, the class of groups considered by D. Arganbright in \cite{DA}, which includes the class of potent $\p$-groups, and the class of groups considered by L.E. Wilson in \cite{LEW}, which includes $\p$-groups of class at most $\p-1$.

Recall that a $\p$-group $\G$ is said to be powerful if $\gamma_2(\G)\subset \G^{\p}$, when $\p$ is odd and $\gamma_2(\G)\subset \G^4$, for $\p=2$. The next lemma can be found in \cite{SM}.

\begin{lemma}\label{L:5.1}
Let $\G$ be a finite $\p$-group. If $\G$ is powerful, then the subgroup $\G^{\p}$ of $G$ is the set of all $\p$th powers of elements of $\G$ and $\G^{\p}$ is powerful.
\end{lemma}

In \cite{LM}, the authors prove the conjecture for powerful groups. The authors of \cite{MHM2} prove that if $N$ is powerfully embedded in $\G$, then the $\e(\MM(\G,\N))\mid \e(\N)$. In the theorem below, we generalize both these results.

\begin{theorem}\label{th:5.2}
Let $\p$ be an odd prime and $\N$ be a normal subgroup of a finite group $\G$. If $\N$ is a powerful $\p$-group, then $\e(\N\wedge \G)\mid \e(\N)$. In particular, $\e(\MM(\G,\N))\mid \e(\N)$.
\end{theorem}

\begin{proof}
Let $\e(\N) =\p^e$. We will proceed by induction on $e$. If $\e(\N) =\p$, then $\gamma_2(\N)\subset \N^{\p} =1$. So by Lemma \ref{L:2.5}$(ii)$, $\e(\N\wedge \G)\mid \e(\N)$. Now let $e>1$. Using Lemma \ref{L:4.5}, consider the exact sequence $\N^{\p}\wedge \G\rightarrow \N\wedge \G \rightarrow \frac{\N}{\N^{\p}}\wedge \frac{\G}{\N^{\p}} \rightarrow 1$. Then $\e(\N\wedge \G)\mid \e(\im(\N^{\p}\wedge \G))\e(\frac{\N}{\N^{\p}}\wedge \frac{\G}{\N^{\p}})$. By Lemma \ref{L:5.1}, $\N^{\p}$ is powerful and $\e(\N^{\p}) =\p^{e-1}$. Thus by induction hypothesis, we have $\e(\N^p\wedge \G)\mid \e(\N^{\p})$, and hence $\e(\im(\N^{\p}\wedge \G))\mid \e(\N^{\p})$. Note that $\frac{\N}{\N^{\p}}$ is a powerful group of exponent $\p$, and we have showed above that the theorem holds for powerful groups of exponent $\p$. Therefore $\e(\frac{\N}{\N^{\p}}\wedge \frac{\G}{\N^{\p}})\mid \p$, whence the result.
\end{proof}

\begin{definition}\label{D.5.3}
Let $\G$ be a finite $p$-group.
\begin{itemize}
\item[$(i)$] We say $\G$ satisfies condition $(1)$ if $\gamma_m(\G)\subset \G^{\p}$ for some $m=2, 3, \dots, \p-1$, and it is said to be a potent $\p$-group if $m=\p-1$.
\item[$(ii)$] We say $\G$ satisfies condition $(2)$ if $\gamma_{\p}(\G)\subset \G^{{\p}^2}$.
\end{itemize}
\end{definition}

The next Theorem can be found in \cite{SM}, \cite{SZ} and \cite{LEW}.

\begin{theorem}\label{th:5.4}
Let $\G$ be a finite $\p$-group. 

\begin{itemize}
\item[$(i)$] If $\G$ regular, then $\G^{\p}$ is the set of all $\p$th powers of elements of $\G$ and $\G^{\p}$ is powerful.
\item[$(ii)$] If $\G$ satisfies condition $(1)$, then $\G^{\p}$ is the set of all $\p$th powers of elements of $\G$ and $\G^{\p}$ is powerful.
\item[$(iii)$] If $\G$ satisfies condition $(2)$, $\G^{\p}$ is the set of all $\p$th powers of elements of $\G$ and $\G^{\p}$ is powerful.
\end{itemize}

\end{theorem}

Groups satisfying condition ($1$) were studied by D. E. Arganbright in \cite{DA}. This class includes the class of powerful $\p$-groups for $\p$ odd and potent $\p$-groups. The groups satisfying condition ($2$) were studied by L. E. Wilson in \cite{LEW}. In the theorem below, we show that the conjecture is true for all these classes of groups. The first part of the Theorem below can be found in \cite{PM4}, but we give a different proof here. Since the proofs of both the parts of the next theorem are similar, we only  prove the second part.

\begin{theorem}\label{th:5.5}
Let $\p$ be an odd prime and $\G$ be a finite $\p$-group satisfying either of the conditions below,
\begin{itemize}
\item[(i)] $\gamma_m(\G)\subset \G^{\p}$ for some  $m$ with $2\leq m\leq \p-1$.
\item[(ii)] $\gamma_{\p}(\G)\subset \G^{\p^2}$.
\end{itemize}
 Then $\e(\G\wedge \G)\mid \e(\G)$. In particular, $\e(\M)\mid \e(\G)$.
\end{theorem}

\begin{proof}
\begin{itemize}
\item[$(ii)$] Let $\e(\G) = \p^n$. If $n =2$, then $\gamma_{\p}(\G)\subset \G^{\p^2} =1$. Therefore by Theorem \ref{th:3.11}, $\e(\G\wedge \G)\mid \e(\G)$. For $n >2$, consider the exact sequence $\G^{\p}\wedge \G\rightarrow \G\wedge \G \rightarrow \frac{\G}{\G^{\p}}\wedge \frac{\G}{\G^{\p}} \rightarrow 1$, which implies $\e(\G\wedge \G)\mid \e(\im(\G^{\p}\wedge \G))\e(\frac{\G}{\G^{\p}}\wedge \frac{\G}{\G^{\p}})$. By Theorem \ref{th:5.4}, $\G^{\p}$ is powerful and $\e(\G^{\p}) =\p^{n-1}$. By Theorem \ref{th:5.2}, $\e(\G^{\p}\wedge \G)\mid \p^{n-1}$, and hence $\e(\im(\G^{\p}\wedge \G))\mid \p^{n-1}$. Since $\gamma_{\p}(\G)\subset \G^{\p^2}\subset \G^{\p}$, the group $\frac{\G}{\G^{\p}}$ has nilpotency class $\le \p-1$. Now applying Theorem \ref{th:3.11}, we obtain $\e(\frac{\G}{\G^{\p}}\wedge \frac{\G}{\G^{\p}})\mid \p$, and hence the result.
 \end{itemize}
\end{proof}

In the next Theorem, we show that to prove the conjecture for regular $\p$-groups, it is enough to prove it for groups of exponent $\p$. We prove it more generally for the exterior square.  
 \begin{theorem}\label{th:5.6}
The following statements are equivalent:
\begin{itemize}
\item[$(i)$] $\e(\G\wedge \G)\mid \e(\G)$ for all regular $\p$-groups $\G$.
\item[$(ii)$] $\e(\G\wedge \G)\mid \e(\G)$ for all groups $\G$ of exponent $\p$.
\end{itemize}

 \end{theorem}
 \begin{proof}
 Since groups of exponent $\p$ are regular, one direction of the proof is trivial. To see the other direction, let $\G$ be a regular $\p$-group. Suppose $\e(\G) = \p^n$, $n >1$. Consider the exact sequence $\G^{\p}\wedge \G\rightarrow \G\wedge \G \rightarrow \frac{\G}{\G^{\p}}\wedge \frac{\G}{\G^{\p}} \rightarrow 1$, which implies $\e(\G\wedge \G)\mid \e(\im(\G^{\p}\wedge \G)) \e(\frac{\G}{\G^{\p}}\wedge \frac{\G}{\G^{\p}})$. By Theorem \ref{th:5.4}, $\G^{\p}$ is powerful and $\e(\G^{\p}) = \p^{n-1}$. Hence Theorem \ref{th:5.2} yields $\e(\G^{\p}\wedge \G)\mid \p^{n-1}$, whence $\e(\im(\G^{\p}\wedge \G))\mid \p^{n-1}$. Since $\frac{\G}{\G^{\p}}$ has exponent $\p$, our hypothesis implies that $\e(\frac{\G}{\G^{\p}}\wedge \frac{\G}{\G^{\p}})\mid \p$. Therefore $\e(\G\wedge \G)\mid \p^n$.

 \end{proof}

\section{Bounds depending on the nilpotency class}

In \cite{GE1}, G. Ellis proves that if the nilpotency class of $\G$ is $\cc$, then $\e(\M)\mid ( \e(\G))^{\ceil{\frac{\cc}{2}}}$. In \cite{PM1}, P. Moravec improves this bound by showing that $\e(\M)\mid ( \e(\G))^{2\floor{\log_2 {\cc}}}$. In the next Theorem, we improve both these bounds. The case $\cc=1$ has been excluded as the conjecture is known to be true in that case.

\begin{theorem}\label{th:6.1}
Let $\G$ be a finite group with nilpotency class $\cc>1$ and set $n= \ceil{\log_{3}(\frac{\cc+1}{2})}$. If $\e(\G)$ is odd, then $\e(\G\wedge \G) \mid (\e(\G))^n$. In particular, $\e(\M)\mid (\e(\G))^n$.
\end{theorem}

\begin{proof}
The proof is by induction on $n$. Note that $n \geq \log_3(\frac{\cc+1}{2})$ if and only if $\cc \leq (3^n \times 2) - 1$. When $n = 1$, the statement follows by Theorem \ref{th:4.6}. Now we proceed to prove it for $n$. Set  $m = \ceil{\frac{\cc+1}{3}}$ and consider the exact sequence $\gamma_m(\G)\wedge \G \rightarrow \G \wedge \G \rightarrow \frac{\G}{\gamma_m(\G) }\wedge \frac{\G}{\gamma_m(\G)} \rightarrow 1$, which can be obtained from Theorem $3.1$ in \cite{ADST}. Thus $\e(\G \wedge \G)\mid \e(\im(\gamma_m(\G) \wedge \G))\e( \frac{\G}{\gamma_m(\G) }\wedge \frac{\G}{\gamma_m(\G)})$. By Lemma \ref{L:2.6}, we obtain that $\e(\im(\gamma_m(\G) \wedge \G)) \mid \e(\gamma_m(\G)) \mid \e(\G)$. Now $\frac{\G}{\gamma_m(\G)}$ is a group of nilpotency class $m-1$. Since $\cc +1\le (3^n \times 2)$, $\frac{\cc+1}{3}\le (3^{n-1}\times 2)$ giving $m\le (3^{n-1}\times 2)$. Now applying induction hypothesis to the group $\frac{\G}{\gamma_m(\G)}$, we obtain $\e( \frac{\G}{\gamma_m(\G) }\wedge \frac{\G}{\gamma_m(\G)} ) \mid (\e(\G))^{n-1}$, which proves the theorem.

\end{proof}

By the previous theorem $\e(\M) \mid (\e(\G))^2$ for a group $G$ with odd exponent and nilpotency class $\cc$ satisfying $6\leq \cc \leq 17$. Moreover, if $G$ is a $\p$-group with $\p\ge \cc$, then $\e(\M) \mid \e(\G)$. In the following theorem, we show that if $\p+1\leq \cc \leq 3\p+2$, then $\e(\M) \mid (\e(\G))^2$.

\begin{theorem}\label{th:6.2}
Let $\p$ be an odd prime and $\G$ be a finite $\p$-group of nilpotency class $\cc$. If $m := \ceil{\frac{\cc+1}{3}}\leq  \p+1$, then $\e(\G \wedge \G) \mid \e(\gamma_m (\G)) \e(\frac{\G}{\gamma_m(\G)})$. In particular, $\e(\M) \mid (\e(\G))^2$.
\end{theorem}

\begin{proof}
Consider the  exact sequence, $\gamma_m(G) \wedge G \rightarrow G\wedge G \rightarrow \frac{G}{\gamma_m(G)}\wedge \frac{G}{\gamma_m(G)} \rightarrow 1$. Thus we have  $\e(G\wedge G) \mid \e(\im(\gamma_m(\G)\wedge \G))\e(\frac{\G}{\gamma_m(\G)}\wedge \frac{\G}{\gamma_m(\G)})$. Now applying Lemma \ref{L:2.6} yields $\e(\im(\gamma_m(\G) \wedge \G)) \mid \e(\gamma_m(\G))$. Further by Theorem \ref{th:3.11}, we have $\e(\frac{\G}{\gamma_m(\G)}\wedge \frac{\G}{\gamma_m(\G)}) \mid \e(\frac{\G}{\gamma_m(\G)})$, and the result follows.

\end{proof}

The next lemma is crucially used in the proof of theorem $\ref{th:6.4}$. In \cite{GE1}, G.Ellis has proved the existence of a covering pair for a pair of groups $(G,N)$. Furthermore, he has given an isomorphism between $[N^*,G]$ and $N\wedge G$, where $N^*$ is a covering pair of the pair $(G,N)$. We use this isomorphism in the following theorem.

\begin{lemma}\label{L:6.3}
Let $\p$ be an odd prime and $\G$ be a finite $\p$-group. If $\N \unlhd \G$ of nilpotency class at most $\p-2$, then $\e(\N \wedge \G) \mid \e(\N)$.
\end{lemma}

\begin{proof}
Consider a projective relative central extension $\delta : N^* \rightarrow G$ associated with the pair $(\G,\N)$. Therefore $\delta(N^*) = N$ and $G$ acts trivially on $\Ker(\delta)$. We know from \cite{GE1} that $[N^*, G] \cong N \wedge G$. Since $N$ is of class atmost $p-2$, $N^*$ is of class at most $p-1$. Set $\e(N)= t$ and since $N^*$ is $t$-central, it is $t$-abelian. Therefore it suffices to prove that ${(n\ ^g{n^{-1}})}^t = 1$ for $n \in N^*$ and $g \in G$. Towards that end, 
 \begin{align*}
 {(n\ ^g{n^{-1}})}^t =& n^t{(^g n^{-1})}^t ,\  \mbox{since}\ N^*\ \mbox{is\ $t$ abelian} \\
 =& n^t\ ^g(n^{-t})\\
 =&1,  \mbox{\ since}\ n^t \in \Ker(\delta) .
  \end{align*}
Hence the result.
\end{proof}

\begin{theorem} \label{th:6.4}
Let $\p$ be an odd prime and $\N$ be a normal subgroup of a finite $\p$-group $\G$. If the nilpotency class of $\N$ is $\cc$, then $\e(\N\wedge \G)\mid (\e(\N))^n$, where $n = \ceil{\log_{\p-1}(\cc+1)}$.
\end{theorem}

\begin{proof}
We proceed by induction on $\cc$. If $\cc\le \p-2$, then by Lemma \ref{L:6.3} $\e(\N\wedge \G)\mid \e(\N)$. Let $\cc > \p-2$ and set $m=\ceil{\frac{\cc+1}{\p-1}}$. By Lemma \ref{L:4.5}, we have an exact sequence $\gamma_m (\N)\wedge \G\rightarrow \N\wedge \G \rightarrow \frac{\N}{\gamma_m (\N)}\wedge \frac{\G}{\gamma_m (\N)} \rightarrow 1$ yielding $\e(\N\wedge \G)\mid \e(\im(\gamma_m (\N)\wedge \G)) \e(\frac{\N}{\gamma_m (\N)} \wedge \frac{\G} {\gamma_m (\N)})$. Since $m(\p-1)\ge \cc+1$, we get $\gamma_{\p-1}(\gamma_m(\N)) =1$. Now using Lemma \ref{L:6.3}, we obtain $\e(\gamma_m(\N) \wedge \G)\mid \e(\gamma_m(\N))$, and hence $\e(\im(\gamma_m(\N)\wedge \G))\mid \e(\N)$. Since $\cc+1\le (\p -1)^n$, we have $\frac{\cc+1}{\p -1}\le (\p -1)^{n-1}$, and hence $m\le (\p -1)^{n-1}$. Now by induction hypothesis, $\e(\frac{\N}{\gamma_m(\N)}\wedge \frac{\G}{\gamma_m(\N)})\mid (\e(\frac{\N}{\gamma_m (\N)}))^{n-1}$. Therefore $\e(\N\wedge \G)\mid (\e(\N))^n$.
\end{proof}

For a finite $\p$-group $\G$ with $\p$ odd, the author of \cite{NS2} proves that $\e(\M)\mid (\e(\G))^m$, where $m=\floor{\log_{\p-1} \cc}+1$. We improve his bound in the next theorem.

\begin{theorem}\label{th:6.5}
Let $\p$ be an odd prime and $\G$ be a finite $\p$-group of nilpotency class $\cc\ge \p$. Then $\e(\G\wedge \G)\mid (\e(\G))^n$, where $n = 1+\ceil{\log_{\p-1} (\frac{\cc+1}{\p+1})}$. In particular, $\e(\M)\mid (\e(\G))^n$.
\end{theorem}

\begin{proof}
 For $\cc =\p$, we have $\e(\G\wedge \G)\mid \e(\G)$ by Theorem $\ref{th:3.11}$. Let $\cc > \p$, and consider the exact sequence $\gamma_{\p+1} (\G)\wedge \G\rightarrow \G\wedge \G \rightarrow \frac{\G}{\gamma_{\p+1} (\G)}\wedge \frac{\G}{\gamma_{\p+1} (\G)} \rightarrow 1$, which yields $\e(\G\wedge \G)\mid \e(\im(\gamma_{\p+1} (\G)\wedge \G)) \e(\frac{\G}{\gamma_{\p+1} (\G)}\wedge \frac{G}{\gamma_{\p+1} (\G)})$. By Theorem \ref{th:3.11}, $\e(\frac{\G}{\gamma_{\p+1} (\G)}\wedge \frac{\G}{\gamma_{\p+1} (\G)})\mid \e(\frac{\G}{\gamma_{\p+1} (\G)})$. Let $k$ be the nilpotency class of $\gamma_{\p+1} (\G)$. Since $\ceil{\frac{\cc+1}{\p+1}} (\p+1)\ge \cc+1$,  we obtain $k+1\le \ceil{\frac{\cc+1}{\p+1}}$. Observe that $\frac{\cc+1} {\p+1} \le (\p-1)^{n-1}$ gives $k+1\le (\p-1)^{n-1}$. Now applying Theorem \ref{th:6.4}, we get $\e(\gamma_{\p+1} (\G)\wedge \G)\mid (\e(\gamma_{\p+1} (\G)))^{n-1}$. Hence $\e(\im(\gamma_{\p+1}(\G)\wedge \G))\mid (\e(\G))^{n-1}$, and the result follows.
\end{proof}

\section{Bounds depending on the derived length}

In \cite{PM1}, the author proved that the conjecture is true for metabelian $\p$-groups of exponent $\p$. In the theorem below, we prove it for $\p$-central metabelian $\p$-groups.
\begin{theorem}\label{th:7.1}
Let $\G$ be a $\p$-central metabelian $\p$-group. Then $\e(\M) \mid \e(\G)$. 
\end{theorem}
\begin{proof}
For groups of exponent $\p$, the theorem holds by \cite{PM1}. Now we consider groups of exponent $\p^{n}$ with $n>1$. Since $\G$ is $\p$-central, we have the commutative diagram
\begin{equation*}
\xymatrix@+20pt{
&\G^{\p}\wedge \G\ar@{->}[r]
\ar@{->}^{\alpha}[d]
 &\G\wedge G\ar@{->}[r]
\ar@{->}[d]
 &\frac{\G}{\G^{\p}}\wedge \frac{\G}{\G^{\p}} \ar@{->}[r]
\ar@{->}^{\beta}[d]
&1 \\
&1\ar@{->}[r]
 &\gamma_2(\G)\ar@{->}[r]
&\gamma_2(\G)\ar@{->}[r]
&1,
} \end{equation*}
where $\alpha$ and $\beta$ are the natural commutator maps.

Now Snake Lemma yields the exact sequence $\ker(\alpha) \rightarrow \M \rightarrow \ker(\beta) \rightarrow 1$. Since $\ker(\beta) \leq \MM(\frac{\G}{\G^{\p}})$,  we have $\e(\MM(\G)) \mid \e(\im(\G^{\p}\wedge \G))\e(\MM(\frac{\G}{\G^{\p}}))$. Observe that $\e(\im(\G^{\p}\wedge \G)) \mid \p^{n-1}$ because $G$ is $\p$-central. Furthermore $\frac{\G}{\G^{\p}}$ being a metabelian $\p$-group of exponent $\p$, $\e(\MM(\frac{\G}{\G^{\p}})) \mid \p$, and the result follows.
\end{proof}

The following Lemma can be deduced from Lemma 3 and the proof of Proposition 5 in \cite{GE2}. 
\begin{lemma}\label{L:7.2}
Let $\N$ be a normal subgroup of a group $\G$. If $\N\subset \gamma_2(\G)$, then the sequence $\N\otimes \G\rightarrow \G\otimes \G \rightarrow \frac{\G}{\N}\otimes \frac{\G}{\N} \rightarrow 1$ is exact.
\end{lemma}

In \cite{PM1}, the author showed that if $\dd$ is the derived length of $G$, then $\e(\M)\mid ( \e(\G))^{2(\dd-1)}$. The author of \cite{NS1} improved this bound for $\p$-groups by proving that $\e(\M)\mid (\e(\G))^{\dd}$ when $\p$ is odd, and $\e(\M)\mid 2^{\dd-1}(\e(\G))^{\dd}$, when $\p=2$. Using our techniques, we obtain the following generalization of Theorem A of \cite{NS1}. 
\begin{theorem}\label{th:7.3}
Let $\G$ be a solvable group of derived length $\dd$. 
\begin{itemize}
\item[$(i)$] If $\e(\G)$ is odd, then $\e(\G\otimes \G)\mid (\e(\G))^{\dd}$. In particular, $\e(\M)\mid (\e(\G))^{\dd}$.
\item[$(ii)$] If $\e(\G)$ is even, then $\e(\G\otimes \G)\mid 2^{\dd-1}(\e(\G))^{\dd}$. In particular, $\e(\M)\mid 2^{\dd-1}(\e(\G))^{\dd}$.
\end{itemize}
\end{theorem}

\begin{proof}
The proof proceeds by induction on $d$. When $\dd =1$, $\G$ is abelian and hence $\e{(\G\otimes \G)}\mid \e{(\G)}$. Let $\dd >1$ and $\G^{(\dd-1)}$ denote the $d$th term of the derived series. Now consider the exact sequence $\G^{(\dd-1)}\otimes \G\rightarrow \G\otimes \G \rightarrow \frac{\G}{\G^{(\dd-1)}}\otimes \frac{\G}{\G^{(\dd-1)}} \rightarrow 1$ obtained from Lemma \ref{L:7.2}. Thus we have $\e(\G\otimes \G)\mid \e(\im(\G^{(\dd-1)}\otimes \G))\e(\frac{\G}{\G^{(\dd-1)}}\otimes \frac{\G}{\G^{(\dd-1)}})$.

\begin{itemize}
\item[$(i)$] Let $\e(\G)$ be odd. Since $\G^{(\dd-1)}$ is abelian, Lemma \ref{L:2.5}$(ii)$ implies $\e(\G^{(\dd-1)}\otimes G)\mid \e(\G^{(\dd-1)})$. Now induction hypothesis yields $\e(\frac{\G}{\G^{(\dd-1)}}\otimes \frac{\G}{\G^{(\dd-1)}})\mid (\e(\frac{\G}{\G^{(\dd-1)}}))^{\dd-1}$, the result follows.

\item[$(ii)$] Now using Lemma \ref{L:2.5}$(iii)$, the proof follows mutatis mutandis the proof of $(i)$.
\end{itemize}
\end{proof}

In the next lemma, we consider $\N \wedge \G$ instead of $\G\otimes \G$.

\begin{lemma}\label{L:7.4}
Let $\N\unlhd \G$. Suppose $\N$ is solvable of derived length $\dd$. 
\begin{itemize}
\item[(i)] If $\e(\N)$ is odd, then $\e(\N\wedge \G)\mid (\e(\N))^{\dd}$. In particular, $\e(\MM(\G,\N))\mid (\e(\N))^{\dd}$.
\item[(ii)] If $\e(\N)$ is even, then $\e(\N\wedge \G)\mid 2^{\dd}(\e(\N))^{\dd}$. In particular, $\e(\MM(\G,\N))\mid 2^{\dd}(\e(\N))^{\dd}$.
\end{itemize}
\end{lemma}

\begin{proof}
We prove the Lemma by induction on $\dd$. For $\dd =1$, the theorem follows from Lemma \ref{L:2.5}$(ii)$ and $(iii)$. Let $\dd >1$, and consider the exact sequence $\N^{(\dd -1)}\wedge \G\rightarrow \N\wedge \G\rightarrow \frac{\N}{\N^{(\dd-1)}}\wedge \frac{\G}{\N^{(\dd-1)}}\rightarrow 1$ obtained from Lemma $\ref{L:4.5}$. Now the proof follows mutatis mutandis the proof of Theorem $\ref{th:7.3}$.

\end{proof}

In the following table, we list and compare the values of $m$ for which $\e(\M) \mid \e(\G)^m$, obtained by G. Ellis, P. Moravec and Theorem \ref{th:6.1} of this paper.
\begin{center}
\textbf{Table I}
\end{center}
\begin{center}
\begin{tabular}{ | m{7em} | m{7em}| m{7em} | m{7em} | m{7em}|}
\hline
    & G. Ellis \cite{GE2} &  P. Moravec \cite{PM1} & \\
 \hline
 $c$ & $\ceil{\frac{\cc}{2}}$ & $2\floor{\log_2\cc}$ & $\ceil{\log_3(\frac{\cc+1}{2})}$\\
 \hline
 3  & 2  & 2  & 1\\
\hline
 4  & 2  & 4  & 1\\
 \hline
 5  & 3  & 4  & 1\\
 \hline
 6  & 3  & 4  & 2\\
  \hline
 17  & 9  & 8  & 2\\
  \hline
 53  & 27  & 10  & 3\\
  \hline
 161  & 81  & 14  & 4\\
  \hline
\end{tabular}
\end{center}

P. Moravec improves the bound given by G. Ellis for $\cc > 11$. It can be seen that the bound obtained in Theorem \ref{th:6.1} improves the other bounds.
In the following table, we consider $\p$-groups of nilpotency class $\cc$ and exponent $\p^n$. The bounds $\p^m$, where $\e(\M) \mid \p^m$, obtained by P. Moravec, N. Sambonet and Theorem \ref{th:6.5} are listed.
\begin{center}
\textbf{Table II}
\end{center}
\begin{center}
\begin{tabular}{ | m{1cm} | m{1cm}| m{1cm} | m{7em} | m{7em}| m{7em} |} 
\hline
&  &  & P. Moravec \cite{PM1} & N. Sambonet \cite{NS2}& \\ 
\hline
$\cc$ & $\p$ & $n$ & $\p^{k\floor{\log_2\cc}}$ & $\p^{n(\floor{\log_{\p-1}\cc}+1)}$ &  $\p^{n(1+\ceil{\log_{\p-1}\frac{\cc+1}{\p+1}})}$\\
 \hline
 5 & 3 & 1 & $3^2$ & $3^3$ & $3^2$ \\
 \hline 
 5 & 3 & 2 & $3^8$ & $3^6$ & $3^4$ \\
 \hline
 7 & 7 & 1 & $7^2$ & $7^2$ & $7$ \\
 \hline
 15 & 13 & 2 & $13^{9}$ & $13^4$ & $13^4$ \\
 \hline
 24 & 5 & 1 & $5^{4}$ & $5^{3}$ & $5^{2}$\\
 \hline
 168 & 13 & 1 & $13^{14}$ & $13^3$ & $13^2$ \\
 \hline
\end{tabular}
\end{center}

where $k$ is defined in \cite{PM1}.

We end this paper by making the following conjecture for which there is no counterexample so far. The two counterexamples to Schur's conjecture found for $2$-groups is not a counterexample to the following conjecture.

\begin{con}
Let $G$ be a finite $p$ group. Then $\e(\M)\mid \p \e(\G)$, i.e exponent of the Schur multiplier divides $\p$ times the exponent of the group.
\end{con}

\section*{Acknowledgements} In a private communication with the third author, P. Moravec had mentioned that for groups with nilpotency class 5, he could prove $\e(\M)\mid (\e(\G))^3$, while the computer evidence indicated that the bound should be 2 instead of 3. Later he himself proved the bound to be 2 in \cite{PM5}. We thank P. Moravec for sharing this insight with us.

\section*{References}
\bibliographystyle{amsplain}
\bibliography{Bibliography}
\end{document}